\documentclass[12pt]{amsart}

\usepackage{tikz-cd}
\usepackage{psfrag}
\usepackage{color}
\usepackage{tikz}
\usetikzlibrary{matrix,arrows}
\usepackage{graphicx,graphics}
\usepackage{fullpage,amssymb,amsfonts,amsmath,amstext,amsthm,amscd,verbatim,enumerate}
\usepackage[T1]{fontenc}

\begin{document}

\newtheorem{theorem}{Theorem}[section]
\newtheorem{result}[theorem]{Result}
\newtheorem{fact}[theorem]{Fact}
\newtheorem{conjecture}[theorem]{Conjecture}
\newtheorem{lemma}[theorem]{Lemma}
\newtheorem{proposition}[theorem]{Proposition}
\newtheorem{corollary}[theorem]{Corollary}
\newtheorem{facts}[theorem]{Facts}
\newtheorem{props}[theorem]{Properties}
\newtheorem*{thmA}{Theorem A}
\newtheorem{ex}[theorem]{Example}
\theoremstyle{definition}
\newtheorem{definition}[theorem]{Definition}
\newtheorem{remark}[theorem]{Remark}
\newtheorem{example}[theorem]{Example}
\newtheorem*{defna}{Definition}

\newcommand{\notes} {\noindent \textbf{Notes.  }}
\newcommand{\note} {\noindent \textbf{Note.  }}
\newcommand{\defn} {\noindent \textbf{Definition.  }}
\newcommand{\defns} {\noindent \textbf{Definitions.  }}
\newcommand{\x}{{\bf x}}
\newcommand{\z}{{\bf z}}
\newcommand{\B}{{\bf b}}
\newcommand{\V}{{\bf v}}
\newcommand{\T}{\mathbb{T}}
\newcommand{\Z}{\mathbb{Z}}
\newcommand{\Hp}{\mathbb{H}}
\newcommand{\D}{\mathbb{D}}
\newcommand{\R}{\mathbb{R}}
\newcommand{\N}{\mathbb{N}}
\renewcommand{\B}{\mathbb{B}}
\newcommand{\C}{\mathbb{C}}
\newcommand{\ft}{\widetilde{f}}
\newcommand{\dt}{{\mathrm{det }\;}}
 \newcommand{\adj}{{\mathrm{adj}\;}}
 \newcommand{\0}{{\bf O}}
 \newcommand{\av}{\arrowvert}
 \newcommand{\zbar}{\overline{z}}
 \newcommand{\xbar}{\overline{X}}
 \newcommand{\htt}{\widetilde{h}}
\newcommand{\ty}{\mathcal{T}}
\renewcommand\Re{\operatorname{Re}}
\renewcommand\Im{\operatorname{Im}}
\newcommand{\tr}{\operatorname{Tr}}
\newcommand{\Stab}{\operatorname{Stab}}

\newcommand{\ds}{\displaystyle}
\numberwithin{equation}{section}

\renewcommand{\theenumi}{(\roman{enumi})}
\renewcommand{\labelenumi}{\theenumi}

\title{Strongly automorphic mappings and Julia sets of uniformly quasiregular mappings}

\author{Alastair Fletcher}
\address{Department of Mathematical Sciences, Northern Illinois University, DeKalb, IL 60115-2888. USA}
\email{fletcher@math.niu.edu}

\author{Douglas Macclure}
\address{Department of Mathematical Sciences, Northern Illinois University, DeKalb, IL 60115-2888. USA}
\email{doug.macclure@gmail.com}
\subjclass[2010]{Primary 37F10; Secondary 30C65, 30D05}
\thanks{This work was supported by a grant from the Simons Foundation (\#352034, Alastair Fletcher).}

\begin{abstract}
A theorem of Ritt states the Poincar\'e linearizer $L$ of a rational map $f$ at a repelling fixed point is periodic only if $f$ is conjugate to a power of $z$, a Chebyshev polynomial or a Latt\`es map. The converse, except for the case where the fixed point is an endpoint of the interval Julia set for a Chebyshev polynomial, is also true.
In this paper, we prove the analogous statement in the setting of strongly automorphic quasiregular mappings and uniformly quasiregular mappings in $\R^n$. Along the way, we characterize the possible automorphy groups that can arise via crystallographic orbifolds and a use of the Poincar\'e conjecture.
We further give a classification of the behaviour of  uniformly quasiregular mappings on their Julia set when the Julia set is a quasisphere, quasidisk or all of $\R^n$ and the Julia set coincides with the set of conical points.
Finally, we prove an analogue of the Denjoy-Wolff Theorem for uniformly quasiregular mappings in $\B^3$, the first such generalization of the Denjoy-Wolff Theorem where there is no guarantee of non-expansiveness with respect to a metric.
\end{abstract}

\maketitle

\section{Introduction and Statement of Results}

\subsection{Linearization} A central theme in dynamics is that of linearization: near a fixed point, a given mapping is conjugated to a much simpler mapping from which the behaviour of the iterates of the mapping near the fixed point can be deduced. In complex dynamics, if $z_0$ is a fixed point of a holomorphic mapping $f$, the multiplier $\lambda = f'(z_0)$ plays a crucial role in determining the behaviour of $f$ near $z_0$. More precisely, the celebrated Koenigs Linearization Theorem states that we can conformally conjugate $f$ to the mapping $w\mapsto \lambda w$ in a neighbourhood of $z_0$ if $|\lambda|$ is not $0$ or $1$. We briefly remark that the case $\lambda =0$ can be dealt with via B\"ottcher's Theorem and the case $|\lambda |=1$ involves more subtlety, but we will say no more about these cases here. We refer to \cite{Milnor} for more detail on linearization in complex dynamics.

If we assume that $|\lambda| > 1$, that is $z_0$ is a repelling fixed point of $f$, then the conformal map which conjugates $f$ to its derivative in a neighbourhood of $z_0$ has an inverse $L$ which can be defined everywhere via the functional equation
\begin{equation}
\label{eq:func} 
f\circ L = L \circ \lambda.
\end{equation}
This map $L$ is a transcendental entire function, satisfies $L(0) = z_0$ and is unique once a value of $L'(0)$ is prescribed. The usual normalization is that $L'(0) = 1$. This map $L$ is called the {\it Poincar\'e linearizer of $f$ at $z_0$}.

One may ask how the properties of the linearizer $L$ depend on $f$. In 1921, Ritt \cite{Ritt} classified the possibilities when $L$ is periodic and $f$ is rational.

\begin{thmA}[\cite{Ritt}]
\label{thm:Ritt}
The Poincar\'e linearizer $L$ of a rational map $f$ at a repelling fixed point $z_0$ is periodic only if $f$ is conjugate to one of:
\begin{enumerate}[(i)]
\item a power map $P_d(z)=z^d$, where $d\geq 2$;
\item a Chebyshev polynomial, that is, a polynomial $T_d$ satisfying $T_d(\cos z) = \cos( dz)$, where $d\geq 2$;
\item a Latt\`es rational map $\ell$ satisfying the functional equation $\ell \circ \Theta = \Theta \circ A$, where $A$ is an affine self-map of a given torus $S$ and $\Theta$ is a holomorphic map from $S$ to the Riemann sphere.
\end{enumerate}
Moreover, the converse is true unless $f$ is a Chebyshev map and we linearize about an end-point of the interval $J(f)$.
\end{thmA}

\begin{example}
\label{ex:1}
To illustrate why the full converse in the theorem above is not true, let $f$ be the Chebyshev polynomial $f(z) = 2z^2-1$. Then if we linearize about the fixed point $z_0 = -1/2$, we obtain the periodic linearizer $L(z) = \cos(z +2\pi/3)$. However, if we linearize about the fixed point $z_0 =1$, which we note is an endpoint of the Julia set $J(f) = [-1,1]$, then we obtain the linearizer $L(z) =\cos \sqrt{z}$. If we insist on the normalized linearizer, we obtain $L(z) = \cosh \sqrt{2z}$.
\end{example}

This example was brought to our attention in \cite{Steinmetz}, and isn't always taken into account in statements of Ritt's result.

Recall that a holomorphic or meromorphic function $h:\C \to \overline{\C}$ can be singly periodic or doubly periodic, but not triply periodic. In the singly periodic case, the function may be invariant under an extra rotation $z\mapsto -z$, such as $\sin z$, or it may not, such as $e^z$ or $\tan z$. The doubly periodic case corresponds to elliptic functions such as the Weierstrass $\wp$ function. Further rotational invariants depend on the choice of the lattice for $\wp$.

In Theorem A, we can classify the three types of periodic linearizers that arise by calling them of exponential-type, sine-type or $\wp$-type.

\subsection{Ritt in higher dimensions}

One of the aims of the current paper is to generalize this result into higher dimensions and the context of quasiregular mappings. Quasiregular mappings in $\R^n$ provide the natural setting for function theory in higher real dimensions. Roughly speaking, quasiregular mappings are mappings with bounded distortion. They are therefore more flexible than holomorphic mappings, but this flexibility is necessary since the generalized Liouville's Theorem denies the existence of conformal mappings in $\R^n$ for $n\geq 3$ that are not M\"obius mappings. Quasiregular mappings share many properties enjoyed by holomorphic mappings, for example there are versions of the Picard and Montel theorems. We refer the reader interested in the foundations of quasiregular mappings to \cite{IM,Rickman,V1}.

Despite this flexibility, it is a non-trivial matter to construct quasiregular mappings. It is a harder task again to construct the closest cousins of holomorphic mappings, the uniformly quasiregular mappings. These will henceforth be called uqr mappings. These are quasiregular mappings for which there is a uniform bound on the distortion of the iterates. Examples corresponding to the cases in Theorem A have been constructed: power mappings \cite{Mayer1}, Chebyshev mappings \cite{Mayer2} and Latt\`es mappings \cite{Mayer1}. These uqr mappings all arise by solving the Schr\"oder functional equation
\[ f\circ L = L\circ \psi\]
given a strongly automorphic quasiregular mapping $L$ and a uniformly quasiconformal mapping $\psi$ satisfying $\psi G \psi^{-1} \subset G$. Briefly, a strongly automorphic mapping is periodic with respect to a group $G$ which acts transitively on fibres. 
The groups we will consider here are all quasiconformal conjugates of discrete groups of isometries in $\R^n$. We will term such quasiconformal groups {\it tame}.
We will make all this more precise below in sections 2 and 3.

On the other hand, if $f$ is a uqr mapping, $x_0$ is a repelling fixed point of $f$ and $\varphi$ is a generalized derivative of $f$ at $x_0$, then there exists a quasiregular linearizer $L$ satisfying the equation
\[ f\circ L = L \circ \varphi.\]
See \cite{HMM,Fletcher1}. One of the aims of this paper is to study the interplay between the Schr\"oder equation and the linearizer equation. Using these notions, our first main result generalizes the aforementioned theorem of Ritt.

\begin{theorem}
\label{thm:1}
Let $n\geq 2$. Suppose $f:\overline{\R^n} \to \overline{\R^n}$ is a uqr map, $x_0\in \R^n$ is a repelling fixed point of $f$ and a linearizer $L:\R^n \to \overline{\R^n}$ for $f$ at $x_0$ is strongly automorphic with respect to a tame quasiconformal group $G$. Then $f$ is one of:
\begin{enumerate}[(i)]
\item a power-type map;
\item a Chebyshev-type map;
\item a Latt\`es-type map.
\end{enumerate}
Conversely, if $h$ is strongly automorphic with respect to the group $G$, $A$ is a uniformly quasiconformal map satisfying $AGA^{-1}\subset G$ and a uqr solution $f$ of the Schr\"oder equation $f\circ h = h\circ A$ is one of the above three types, then a linearizer $L$ for $f$ at $h(0)$ is strongly automorphic if $0$ is not in the branch set of $h$. If $0$ is in the branch set of $h$, and $\Stab(0)$ is the subgroup of $G$ fixing $0$, then $L \circ P$ is strongly automorphic for some quasiregular map $P$ which is strongly automorphic with respect to $\Stab(0)$.
\end{theorem}

The final statement in this theorem deals with the generalization of the phenomenon observed in Example \ref{ex:1}.
We will be more precise about the various definitions here in sections 2 and 3. One of our tasks will be to determine the allowable types of groups $G$ with respect to which $L$ is strongly automorphic. We will see in section 3 that crystallographic groups play a key role here, and in particular those crystallographic groups for which the underlying space of the crystallographic orbifold is a topological sphere. The linearizer $L$ in the three cases corresponds to an analogue of an exponential-type map, sine-type map and $\wp$-type map respectively.

\begin{remark}
A tame quasiconformal group $G$ is quasiconformally conjugate to a discrete group $G'$ of isometries in $\R^n$. As we will see below, work of Martio and Srebro \cite{MS} implies that the translation subgroup of $G'$ must have rank $n-1$ or $n$. On the other hand, there do exist periodic quasiregular mappings in $\R^n$ where the translation subgroup has any rank $k\in \{1,\ldots,n\}$. By work of Martio \cite{Martio}, a generic point has infinitely many pre-images in a fundamental domain for the action of the group if $k\leq n-2$. Consequently, in this case there can be no $f:\overline{\R^n} \to \overline{\R^n}$ having such a map as a linearizer, since every such $f$ has finite degree.
\end{remark}

Next, it follows from Ritt's Theorem that in the holomorphic case, if $L$ is periodic, then the Julia set of $f$ is either all of $\overline{\C}$ in the Latt\`es case, or contained in a generalized circle in the power and Chebyshev cases. In fact, a result of Fatou \cite[Section 43]{Fatou} states that if the Julia set of a rational map is a smooth curve then it is contained in a circle. More recently, it was proved \cite{BE0,EV} that if the Julia set of a rational map is contained in a smooth curve, it is contained in a circle. 

In the quasiregular setting, we cannot hope for exact analogues of such results since we are free to conjugate by quasiconformal mappings. However, Theorem \ref{thm:1} does give the following.

\begin{corollary}
\label{cor:1}
Let $n\geq 2$.
If a quasiregular map $L:\R^n \to \overline{\R^n}$ is strongly automorphic with respect to a tame quasiconformal group $G$ and $A:\R^n \to \R^n$ is a uniformly quasiconformal map satisfying $AGA^{-1} \subset G$, then the Julia set of the uqr solution $f$ of the Schr\"oder equation $f\circ L = L\circ A$ is either all of $\overline{\R^n}$, an $(n-1)$-quasisphere or an $(n-1)$-quasidisk.
\end{corollary}

Here, an $(n-1)$-quasisphere is the image of the unit sphere $S^{n-1}$ in $\R^n$ under an ambient quasiconformal map of $\R^n$ and an $(n-1)$-quasidisk is the image of the disk $ D_{n-1} = \{ (x_1,\ldots,x_n) : x_n = 0, x_1^2+ \ldots + x_{n-1}^2 \leq 1 \}$ under an ambient quasiconformal map of $\R^n$.

The methods of proof employed in Corollary \ref{cor:1} can also be used to construct a quasiregular analogue of the rational map $R(z) = (z+1/z)/2$. We observe that this generalizes the equation $\cosh(z) = R(e^z)$.

\begin{theorem}
\label{thm:zz}
Let $U,V \subset \R^n$ be an $(n-1)$-quasisphere and an $(n-1)$-quasidisk respectively. Then there exists a degree two quasiregular map $h_1:\overline{\R^n} \to \overline{\R^n}$ such that $h_1(U) = V$. Moreover, there exists a quasiconformal involution $\rho :\overline{\R^n} \to \overline{\R^n}$ which switches the components of $\overline{\R^n} \setminus U$ and satisfies $h_1 \circ \rho = h_1$.
\end{theorem}

\begin{remark}
\begin{enumerate}[(i)]
\item There is plenty of flexibility in this construction and by some judicious choices, one can obtain a quasiregular map $h_1:\overline{\R^3} \to \overline{\R^3}$ which extends the rational map $(z+1/z)/2$. Clearly then one can obtain a quasiregular extension of $z+1/z$, however we do not know if this extension is uqr, c.f. a question of Martin in \cite{Martin}.
\item The rational map $z+1/z$ is called the Joukowsky transform in the context of aerodynamics and solutions of the potential flow over an aerofoil shape. It could be of interest to investigate potential applications of higher dimensional quasiregular versions of the Joukowsky transform.
\item While there is no guarantee that the distortion of this map is smaller than the degree, $2$, a modification of the construction allows one to increase the degree without increasing the distortion. This yields an analogue of $(z^d + 1/z^d)/2$. The point here is that one could study the dynamics of this mapping since a Julia set is defined only when the degree is larger than the distortion (see \cite{Berg1}).
\end{enumerate}
\end{remark}

\subsection{Dynamics on quasiballs}

In \cite{EV} a classification is given of all rational maps whose Julia set is contained in a circle. One may therefore ask to what extent Corollary \ref{cor:1} has a converse. Unfortunately a complete classification of uqr mappings which have Julia set equal to, for example, $S^{n-1}$ is out of reach since a uqr map may be modified on a Fatou component without changing the dynamics on the Julia set. See the argument in Proposition \ref{prop:large} for an example of such a modification. In this direction, it was proved in \cite{MM} that if the Julia set of a uqr map is $S^{n-1}$ and agrees with the set of conical points $\Lambda(f)$ of $f$, that is the set of points where a linearization can be performed, then the restriction of $f$ to the sphere $J(f)$ is a Latt\`es-type map. We give an extension of this result to cover the three cases in Corollary \ref{cor:1}.

\begin{theorem}
\label{thm:2}
If $f:\overline{\R^n} \to \overline{\R^n}$ is a uqr map with Julia set either all of $\overline{\R^n}$, for $n\geq 3$, or an $(n-1)$-quasi-sphere or an $(n-1)$-quasi-disk with $n\geq 4$ and $\Lambda (f) = J(f)$, then $f$ restricted to its Julia set agrees with either a Latt\`es-type map, a power-type map or a Chebyshev-type map respectively.
\end{theorem}

The remaining cases here, that of the $2$-quasisphere or $2$-quasidisk, run into the problem of determining those rational maps with Julia set equal to $\overline{\C}$ and having a uqr extension to $\overline{\R^3}$. This remains an open question.

While a complete classification of uqr maps which have $S^{n-1}$ as their Julia set is not possible, one can still ask for a classification of the dynamical behaviour of a uqr map in $\R^n$ which has a ball as a completely invariant set or, more generally, the behaviour of the iterates of a uqr map $f:\B^n \to \B^n$, where $\B^n$ denotes the open unit ball in $\R^n$. In the unit disk $\D \subset \C$, a complete answer to this question is given by the Denjoy-Wolff Theorem: either $f$ is an elliptic M\"obius map and the iterates form a semi-group of automorphisms of $\D$, or there is a unique $z_0 \in \overline{\D}$ so that the iterates of $f$ converge uniformly on compact subsets of $\D$ to $z_0$.

Various generalizations of the Denjoy-Wolff Theorem abound, for example in the setting of Hilbert spaces and Banach spaces, picking just one recent example \cite{LLNW}. However, as far as the authors are aware, all generalizations of the Denjoy-Wolff Theorem rely on the existence of a metric with respect to which the mapping is non-expansive. In the uqr setting, this presents an issue since uqr mappings can increase hyperbolic distances by an arbitrarily large factor, see Proposition \ref{prop:large}. In dimension two, every uqr map is conjugate to a holomorphic map by a result of Sullivan \cite{Sullivan} and Tukia \cite{Tukia}. Consequently, the Denjoy-Wolff Theorem for uqr maps in dimension two follows easily. 

Roughly speaking, a uqr map in dimension two turns out to be non-expansive with respect to a metric obtained by modifying the hyperbolic metric through a quasiconformal map arising from a solution of the Beltrami differential equation. This latter property no longer holds in higher dimensions. More precisely, while every uqr map in $\R^n$ for $n \geq 3$ is rational with respect to some invariant conformal structure, this structure cannot necessarily be integrated to give a quasiconformal map, see for example \cite{IM}. It is therefore not clear if there is a suitable metric with respect to which a uqr map in $\B^n$ for $n\geq 3$ is non-expansive.

In light of this, it is somewhat surprising that we are able to obtain the following version of the Denjoy-Wolff Theorem in dimension $3$ which is, as far as we are aware, the first version of the Denjoy-Wolff Theorem which does not directly use non-expansiveness with respect to a metric.

\begin{theorem}
\label{thm:3}
Let $f:\B^3 \to \B^3$ be a surjective, proper uqr map. Then either the family of iterates of $f$ forms a semi-group of automorphisms of $\B^3$, or there exists a unique $x_0 \in \overline{\B^3}$ such that the iterates of $f$ converge uniformly on compact subsets of $\B^3$ to $x_0$.
\end{theorem}

In view of this result, we call $x_0$ the Denjoy-Wolff point of $f$. Clearly, we may replace $\B^3$ by any quasiball in $\R^3$.
Using the fact that if a Fatou component $U$ of a uqr map contains an attracting fixed point then $\partial U = J(f)$, we immediately have the following corollary.

\begin{corollary}
\label{cor:2}
Let $f:\R^3 \to\R^3$ be a non-injective uqr map. If $U$ is a completely invariant quasiball, then $J(f) \subset \partial U$. If the Denjoy-Wolff point $x_0$ of $f|_U$ is contained in $U$, then $J(f) = \partial U$, whereas if $x_0 \in \partial U$ and is an attracting fixed point, then $J(f)$ is a proper subset of $\partial U$.
\end{corollary}

\begin{remark}
\begin{enumerate}[(i)]
\item Our method takes advantage of the fact that the hypotheses imply $f$ extends to $S^2$, the restriction to $S^2$ is a uqr map and hence is conjugate to a rational map. We are therefore unable to drop surjectivity and properness from the assumptions.
\item Martin \cite{Martin} gave some severe non-existence results on the rational maps which have a uqr extension to $\B^3$. This therefore restricts the possibilities in Theorem \ref{thm:3}. For example, the extension of $f$ to $S^2$ cannot have a superattracting cycle, a cycle of a Siegel disk or Herman ring, nor two cycles which are attracting or rationally indifferent.
\item Our method could be extended to much more generality if the following question of Aimo Hinkkanen could be answered in the negative: can a non-injective uqr map ever have a continuum (consisting of more than one point) of fixed points? The main results in this direction are contained in \cite{Martin02}. 
\item Power-type mappings give examples where the Denjoy-Wolff point is in $\B^3$, but the authors are unaware of any examples where the Denjoy-Wolff point is in $\partial \B^3$. Do uqr analogues of hyperbolic or parabolic Blaschke products exist?
\item In the case where the Denjoy-Wolff point is in the boundary of a completely invariant quasiball, and is also attracting, must $J(f)$ be a Cantor subset of $\partial U$? Recall that for a Blaschke product $B:\overline{\C} \to \overline{\C}$ which has $\D$ as a completely invariant domain, $J(B)$ is either the unit circle, or a Cantor subset of it.
\end{enumerate}
\end{remark}

\subsection{Organisation of the paper}

The paper is organised as follows. In section 2, we recall some basic facts about quasiregular mappings, quasiregular dynamics  and crystallographic groups. In section 3, we will analyze the Schr\"oder equation and classify the sorts of crystallographic groups that can occur for strongly automorphic quasiregular mappings. In section 4, we will prove Theorem \ref{thm:1} and Corollary \ref{cor:1}. In section 5, we will point out how the methods used in the proof of Theorem \ref{thm:1} also allow one to construct a quasiregular version of the rational map $(z+1/z)/2$ and prove Theorem \ref{thm:zz}. In section 6, we will prove Theorem \ref{thm:2} and finally in section 7 we will prove Theorem \ref{thm:3}.

The authors wish to thank Aimo Hinkkanen for interesting conversations, Dan Grubb for help formulating and proving the topological result in Lemma \ref{lem:dw4} and in particular Dan Nicks for thoroughly reading a previous draft.

\section{Preliminaries}

Throughout, we will denote by $B(x_0,r)$ the open Euclidean ball of radius $r>0$ centred at $x_0 \in\R^n$. We will denote by $x=(x_1,\ldots, x_n)$ an element of $\R^n$. The standard unit basis vectors in $\R^n$ are denoted by $e_1,\ldots, e_n$.

\subsection{Quasiregular mappings}

A {\it quasiregular mapping} in a domain $U\subset \R^n$ for $n\geq 2$ is a continuous mapping in the Sobolev space $W^1_{n,loc}(U)$ where there is a uniform bound on the distortion, that is, there exists $K\geq 1$ such that
\[|f'(x)|^n \leq KJ_f(x)\]
almost everywhere in $U$. The minimum such $K$ for which this inequality holds is called the {\it outer dilatation} and denoted by $K_O(f)$. As a consequence of this, there is also $K' \geq 1$ such that 
\[J_f(x) \leq K' \inf_{|h|=1}|f'(x)h|^n\]
holds almost everywhere in $U$. The minimum such $K'$ for which this inequality holds is called the {\it inner dilatation} and denoted by $K_I(f)$. If $K= \max \{K_O(f), K_I(f) \}$, then $K=K(f)$ is the {\it maximal dilatation} of $f$. A $K$-quasiregular mapping is a quasiregular mapping for which $K(f) \leq K$. 
The set of points where a quasiregular mapping $f$ is not locally injective is called the branch set, and denoted $B_f$.
An injective quasiregular mapping is called quasiconformal. A quasiregular mapping with poles is sometimes called quasimeromorphic, but we will retain the nomenclature quasiregular (c.f. quasiregular mappings between manifolds, where here the manifolds are $\R^n$ and $S^n$).

We refer to \cite{Rickman} for many more details on the foundations of quasiregular mappings, but we note here that a quasiregular mapping is open, discrete and orientation preserving. We will also use the fact that quasiregular mappings have bounded linear distortion. If we let 
\[L(x_0,r,f) := \max_{|x-x_0|=r}|f(x)-f(x_0)| \text{ and }l(x_0,r,f):= \min_{|x-x_0|=r}|f(x)-f(x_0)|,\]
then the {\it linear distortion} of $f$ at $x_0$ is 
\[ H(x_0,f) = \limsup_{r \to 0} \frac{ L(x_0,r,f) }{l(x_0,r,f)}.\]

The fact that the linear distortion is uniformly bounded for quasiregular mappings is important for numerous applications.

\begin{theorem}[\cite{Rickman}, Theorem II.4.3]
\label{thm:lindist}
There exists a constant $C$ depending only on the dimension $n$ and the product $i(x,f)K_O(f)$ of the local index of $f$ at $x$ and the outer distortion so that
\[ H(x,f) \leq C.\]
\end{theorem}

We will also need to mention bounded length distortion mappings, or BLD for short. These are quasiregular mappings with the extra condition that they are locally Lipschitz.

\subsection{Iterates of quasiregular mappings and uqr mappings}

For $m\geq 1$, we write $f^m$ for the $m$-fold iterate of $f$. A mapping is called {\it uniformly $K$-quasiregular}, or $K$-uqr for short, if $K(f^m) \leq K$ for all $m\geq 1$.

It follows from Miniowitz's version of Montel's Theorem \cite[Theorem 4]{Miniowitz} that for a uqr map $f:\overline{\R^n} \to \overline{\R^n}$, space breaks up into the Fatou set and Julia set in exact analogy with complex dynamics in the plane. More precisely, $x_0$ is in the Fatou set $F(f)$ if there exists a neighbourhood $U$ of $x_0$ on which the family of iterates forms a normal family. Conversely, $x_0$ is in the Julia set $J(f)$ if no such neighbourhood can be found. A useful characterization of the Julia set is via the blowing-up property: $x_0 \in J(f)$ if and only if for every open set $U$ containing $x_0$, the forward orbit $O^+(U) = \cup_{m\geq 0}f^m(U)$ contains all of $\R^n$ except possibly for finitely many points.
For an introduction to the theory of quasiregular iteration, we refer to \cite{Berg}.

The fixed points of a uqr mapping can be classified in a similar way to those for holomorphic mappings in the plane. We refer to \cite{HMM} for a complete classification, but we briefly discuss the case of repelling fixed points, since it is of relevance to this paper. An immediate issue with classifying fixed points is that a uqr map need not be differentiable at a given fixed point. To deal with this problem, Hinkkanen, Martin and Mayer \cite{HMM} introduced the notion of a generalized derivative obtained as a limit of
\[ \frac{f(x_0+\lambda_m x) - f(x_0) }{\lambda_m} \]
through a sequence $\lambda_m \to 0$. While for a given sequence the limit may not exist, it is guaranteed to along a subsequence by the normal family machinery.
It is possible for $f$ to have more than one generalized derivative at $x_0$, and the collection of generalized derivatives is called the infinitesimal space of $f$ at $x_0$. A generalized derivative is a uniformly quasiconformal mapping of $\R^n$.

By \cite[Lemma 4.4, Definition 4.5]{HMM}, a fixed point $x_0$ of $f$ is called {\it repelling} if one, and hence all, generalized derivative $\varphi$ of $f$ is loxodromic, that is, fixes $0$ and $\infty$ and so that $\varphi^m(x) \to \infty$ for all $x\neq 0$. 
By \cite[Theorem 6.3 (ii)]{HMM}, a uqr mapping $f$ can be linearized by a quasiregular mapping $L$ of transcendental type (that is, the degree is infinite) to a generalized derivative $\varphi$. 

\begin{definition}
Let $f$ be a uqr map with repelling fixed point $x_0$, and let $\varphi$ be a generalized derivative of $f$ at $x_0$. Then a quasiregular map $L:\R^n \to \overline{\R^n}$ is called a quasiregular linearizer for $f$ at $x_0$ if $L(0)=x_0$, $L$ is locally injective near $0$ and $L$ satisfies the equation
\[ f\circ L = L\circ \varphi.\]
\end{definition}

Given $f,x_0$ and $\varphi$, there are infinitely many possibilities for $L$. In the plane, a linearizer is normalized via $L'(0)=1$. Note that such a normalization is unavailable for quasiregular mappings.
The function theoretic and dynamical properties of $L$ and the dependence on $f$ and $\varphi$ were studied in \cite{Fletcher1} and \cite{FM}.

\subsection{Crystallographic groups}

We first recall several topological notions. A convex polytope in $\R^n$ is the convex hull of a finite collection of points. Every convex polytope has an associated dimension, which may be strictly less than the dimension of the ambient space. For brevity, in this paper we will consider only convex polytopes, and drop the term convex.
An $n$-cell $\Omega$ in $\R^n$ is a set of the form
\[ \Omega = \{ (x_1 , \ldots, x_n ) : a_i \leq x_i \leq b_i, i \in \{1,\ldots, n\} \},\]
where $a_i <b_i$ for $i\in \{1,\ldots, n\}$.

A group $G$ acting on $\R^n$ is called a crystallographic group if it is a discrete cocompact group of isometries. We refer to \cite{Sz} for more information on crystallographic groups than the outline we present here.

Every crystallographic group has a maximal subgroup of translations $T$ generated by $n$ linearly independent vectors $v_1,\ldots,v_n$. The point group $P=G/T$ is a finite group. 
A fundamental set $\Omega$ for $G$ is a subset of $\R^n$ containing precisely one point from each of the orbits of $G$. We may assume that $\Omega$ is connected, and thus is the interior of an $n$-polytope together with some of its boundary. 
The identifications of points of $\partial \Omega$ comes from $G$.
A fundamental domain is then $\operatorname{int}(\Omega)$. See \cite[\S 5.1]{MS} for more on this distinction between a fundamental domain and a fundamental set.

Since $G$ is cocompact, the quotient space $\mathcal{K} = \R^n / G$ is compact as a topological space. 
Moreover, $\mathcal{K}$ is an orbifold, that is, it is locally modelled on the quotient of Euclidean space under the action of a finite group.

\begin{definition}
The crystallographic orbifold $\mathcal{K}$ is called {\it spherical} if the underlying space of $\mathcal{K}$ is a topological sphere.
\end{definition}

If the point group is just the identity, then $\mathcal{K}$ is an $n$-torus and hence not spherical. We necessarily must have $P$ generated by rotations for $\mathcal{K}$ to be spherical. In dimension two, there are $17$ crystallographic groups and one can check that four of them have spherical orbifolds. One can further consider the cases in dimension $3$ by analyzing the classification of Dunbar \cite{Dunbar}.

Since we will be using spherical orbifolds as a classification, it is worth pointing out that there do exist spherical orbifolds in every dimension.

\begin{example}
\label{ex:2}
Consider the translation group $T$ acting on $\R^n$ generated by $< x\mapsto x+e_i : i=1,\ldots ,n >$. Choose the fundamental set $\Omega_1$ for the action of $T$ so that $\overline{\Omega_1}$ is the $n$-cell $[-1/2,1/2]^n$.
Now consider the subgroup $P$ of $O(n)$ generated by rotations through $\pi$ which have fixed point set equal to one of the $(n-2)$-dimensional hyperplanes generated by $n-2$ of the standard basis vectors in $\R^n$.

Letting $G$ be the crystallographic group generated by $T$ and $P$, we can choose the fundamental set $\Omega$ for $G$ so that $\overline{\Omega}$ is the $n$-cell $[-1/2,1/2]\times [0,1/2]^{n-1}$. We can view $\overline{\Omega}$ as a union of two closed topological balls, given by $B_1 = \overline{\Omega} \cap \{x_1 \geq 0 \}$ and $B_2 = \overline{\Omega} \cap \{x_1 \leq 0 \}$ respectively, which are glued together along their respective boundaries under the action of $G$ and the obvious gluing in the $x_1 = 0$ hyperplane.

For any face $F$ of $\overline{\Omega}$ in a hyperplane of the form $x_j=0$ or $x_j = 1/2$ for $j=2,\ldots, n$, the action of $G$ glues the part with $x_1\geq 0$ to the part with $x_1 \leq 0$ by, informally, folding over the $(n-2)$-dimensional set $F \cap \{x_1= 0\}$. The faces contained in hyperplanes of the form $x_1 = \pm 1/2$ are glued together by an element of $T$. Consequently, $B_1$ and $B_2$ are glued toegether in such a way that we obtain a topological sphere as the underlying space of the crystallographic orbifold.

See Figure \ref{fig:1} for the dimension two version of this.

\begin{figure}[h]
\begin{center}
\includegraphics[width=6in]{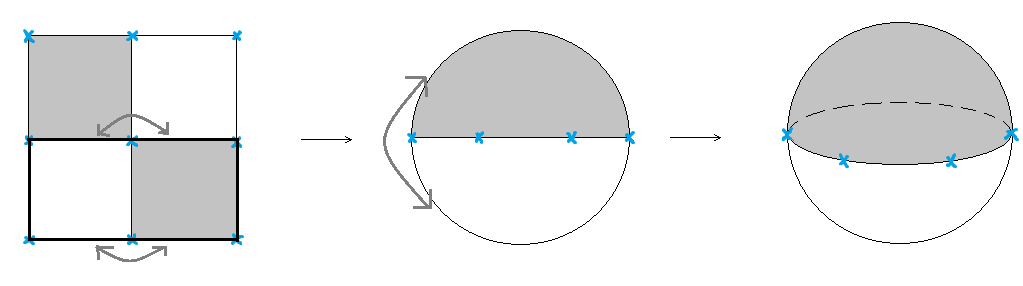}
\caption{Identifying the square to give a $2$-sphere.}
\label{fig:1}
\end{center}
\end{figure}

\end{example}

We remark that this example is surely standard to topologists, but we were unable to find a specific reference in the crystallographic literature. The theme of this example is gluing together two topological $n$-disks along their boundary in a standard way to obtain the usual $n$-sphere $S^n$. Gluing via more complicated homeomorphisms leads to the theory of exotic spheres.

\section{Strongly automorphic quasiregular mappings}

\subsection{The definition}

\begin{definition}
A quasiregular mapping $h:\R^n \to \overline{\R^n}$ is called {\it strongly automorphic} with respect to a quasiconformal group $G$
if the following two conditions hold:
\begin{enumerate}[(i)]
\item $h\circ g = h$ for all $g\in G$,
\item $G$ acts transitively on the fibres $h^{-1}(y)$, that is, if $h(x_1) = h(x_2)$, then there exists $g\in G$ such that $x_2=g(x_1)$.
\end{enumerate}
\end{definition}

The following definition encapsulates the groups to be considered in this paper.

\begin{definition}
A quasiconformal group $G$ acting on $\R^n$ is called {\it tame} if there exists a quasiconformal mapping $\varphi : \R^n \to \R^n$ and a discrete group of isometries $G'$ acting on $\R^n$ so that $G = \varphi G'\varphi^{-1}$
\end{definition}

\begin{remark}
\begin{enumerate}[(i)]
\item The group $G$ must necessarily be discrete if $h$ is required to be non-constant, since quasiregular mappings are discrete.
\item Strongly automorphic mappings in the literature have been considered when $G$ itself is a discrete group of isometries. The generalization here to quasiconformal conjugates of such groups does not lead to a much wider class of quasiregular mappings. In fact, Lemma \ref{lem:comp} below shows that a quasiregular map which is strongly automorphic with respect to a tame quasiconformal group is quasiconformally conjugate to a strongly automorphic quasiregular map under the previous definition. 
\item On the other hand, we do need to consider this wider class to prove our results, and in particular apply \cite[Corollary 3.7]{FM} in proving a linearizer is strongly automorphic in the proof of Theorem \ref{thm:1} below. The reason for this is that a generalized derivative $\varphi$ is just a uniformly quasiconformal map and typically will not satisfy $\varphi G\varphi^{-1} \subset G$ for a discrete group of isometries $G$.
\item Not every quasiconformal group is tame, see for example \cite{Tukia2}. We leave the exploration of the case where $G$ is a quasiconformal group that is not tame to future work.
\item It is worth remarking that the second condition means that not every periodic function is strongly automorphic. For example, $f(z) = e^z$ is strongly automorphic with respect to the group of translations generated by $w\mapsto w+ 2\pi i$, but it is not hard to see that $f(z) = e^{e^z}$ is not strongly automorphic with respect to a group of translations because the second condition does not hold. The set $f^{-1}(e)$ consists of points of the form 
\[ \frac{1}{2} \ln (1+4k^2\pi^2) + i \left [ \tan^{-1} (2k\pi) + 2m \pi \right ], \]
for $k,m\in \Z$. This is not a quasiconformal image of a lattice in $\C$. With regards to Theorem A, $e^{e^z}$ cannot be a linearizer for a rational map, since every linearizer has finite order, depending on the multiplier of the fixed point and the degree of $f$.
\end{enumerate}
\end{remark}

\subsection{The Schr\"oder equation}

Every strongly automorphic mapping has a family of associated uqr mappings.

\begin{theorem}
\label{thm:schroder}
Suppose $h:\R^n \to \overline{\R^n}$ is strongly automorphic with respect to a tame quasiconformal group $G$, that is $G= \varphi G' \varphi^{-1}$ where $G'$ is a discrete group of isometries and $\varphi$ is a quasiconformal map. Further suppose that there is a uniformly quasiconformal mapping $A$ satisfying $A(0)=0$ and 
\[ AGA^{-1} \subset G.\]
Then there is a unique uqr map $f:\overline{\R^n} \to \overline{\R^n}$ which solves the Schr\"oder equation
\[ f\circ h = h\circ A.\]
\end{theorem}

The proof of this theorem is almost identical to \cite[Theorem 21.4.1]{IM}, where $G$ is a discrete group of isometries and $A$ is a uniformly quasiconformal linear map, but we include a proof for the convenience of the reader.

\begin{proof}
The first step is to show that $f$ is well-defined on $h(\R^n)$. If $h(x_1) = h(x_2)$, then since $h$ is strongly automorphic with respect to $G$, there exists $g \in G$ such that $g(x_1) = x_2$. By the hypotheses of the theorem, there is $g_1 \in G$ such that $A \circ g = g_1 \circ A$. Therefore
\[ h(A(x_2))=  h(A(g(x_1))) = h(g_1(A(x_1))) = h(A(x_1)).\]
Therefore $f$ is well-defined and continuous on $h(\R^n)$. Since $h$ can omit at most two values in $\overline{\R^n}$ (see \cite{MS}), the omitted values of $h$ are the exceptional values of $f$. Hence $f$ is well-defined in $\overline{\R^n}$.

To see that $f$ is quasiregular, away from the post-branch set $h(B_h)$ of $h$ we can write $f=h\circ A\circ h^{-1}$ for a suitable branch of the inverse. Then since $h$ is quasiregular and $A$ is quasiconformal, $f$ is $K(h)^2K(A)$-quasiregular away from $h(B_h)$. Since $f$ is continuous and the branch set of $h$ has measure zero, we conclude that $f$ is in fact quasiregular. To see that $f$ is unique, $h\circ A \circ h^{-1}$ is uniquely defined in a neighbourhood of $h(0)$ for any choice of branch of $h^{-1}$. The functional equation then implies $f$ is uniquely defined everywhere.

The Schr\"oder equation implies that for every $m\in \N$ we have 
\[ f^m \circ h = h \circ A^m.\]
Since $A$ is uniformly quasiconformal, the distortion of the right hand side is bounded above over all $m\in \N$. Consequently $f$ is uniformly quasiregular. 
\end{proof}

\subsection{The three cases}

The next question is to classify the groups that can give rise to strongly automorphic quasiregular mappings. 
Suppose that $h$ is strongly automorphic with respect to $G$, where $G = \varphi G' \varphi^{-1}$, $\varphi$ is quasiconformal and $G'$ is a discrete group of isometries. The next lemma shows that we may pass to $\varphi^{-1} \circ h \circ \varphi$, which is strongly automorphic with respect to $G'$.

\begin{lemma}
\label{lem:comp}
Suppose $h$ is a strongly automorphic quasiregular mapping with respect to $G$. 
\begin{enumerate}[(i)]
\item If $p:\overline{\R^n} \to \overline{\R^n}$ is quasiconformal, then $p\circ h$ is strongly automorphic with respect to $G$.
\item If $p:\R^n \to \R^n$ is quasiconformal, then $h\circ p$ is strongly automorphic with respect to $p^{-1}Gp$.
\end{enumerate}
\end{lemma}

\begin{proof}
The first part is clear. For the second part, if $g \in G$ then
\[ (h \circ p) \circ (p^{-1} \circ g \circ p) = h\circ g \circ p = h \circ p,\]
and if $h(p(x_1)) = h(p (x_2))$ then $p(x_2) = g(p (x_1))$ for some $g\in G$. Then $x_2 = g'(x_1)$ for some $g' \in p^{-1}Gp$. This shows that $h$ satisfies the two conditions to be strongly automorphic.
\end{proof}

We may therefore assume that 
\begin{equation}
\label{eq:h1} 
h_1 = \varphi^{-1} \circ h \circ \varphi
\end{equation}
is strongly automorphic with respect to a discrete group of isometries $G'$. The group $G'$ contains a maximal subgroup $T$ consisting of translations and by the Bieberbach theory, see for example \cite{Sz}, the quotient $G'/T$ is a finite group. In fact, in our case, $G'/T$ must consist of rotations since $h\circ g$ is orientation preserving for every $g\in G'$.

As was observed in \cite{MS}, not every periodic quasiregular mapping is strongly automorphic. In particular, there are only two possibilities for the translation subgroup: $T$ must be isomorphic to either $\Z^n$ or $\Z^{n-1}$. While one can construct periodic quasiregular mappings with $T$ isomorphic to $\Z^k$ for $k<n-1$, the Schr\"oder equation no longer has a guaranteed uqr solution.

If $T$ is isomorphic to $\Z^n$, then $G'$ is a crystallographic group acting on $\R^n$.
In this case, the quasiregular map must have poles (see \cite{MS}). We will call such quasiregular mappings of $\wp$-type, in analogy with the doubly periodic Weierstrass $\wp$-function in the plane.

The second case, where $T$ is isomorphic to $\Z^{n-1}$ and $G'|_{\R^{n-1}}$ acts as a crystallographic group on $\R^{n-1}$, splits into two further subcases. 
A fundamental domain for the action of $T$ on $\R^n$ can be taken to be a beam in $\R^n$ that is, without loss of generality, perpendicular to $\{x_n = 0\}$, and the quotient $G'/T$ either contains a rotation identifying the two ends of the beam, or it doesn't. 
We call the quasiregular mappings arising in the first subcase sine-type, in analogy with trigonometric functions in the plane, and the second subcase are called Zorich-type, since Zorich \cite{Zorich} was the first to construct such quasiregular generalizations of the exponential function.

We can therefore classify strongly automorphic mappings, and uqr solutions to a Schr\"oder equation involving the strongly automorphic map, as follows.

\begin{definition}
Let $h$ be a strongly automorphic quasiregular mapping with respect to a tame quasiconformal group $G$, let $\varphi$ be a quasiconformal map with $G = \varphi G' \varphi^{-1}$ where $G'$ is a discrete group of isometries, let $A$ be a loxodromically repelling uniformly quasiconformal map with $AGA^{-1} \subset G$ and let $f$ be a uqr solution to the Schr\"oder equation $f\circ h = h \circ A$.
We say that:
\begin{enumerate}
\item $h$ is of $\wp$-type if the translation subgroup $T$ of $G'$ is isomorphic to $\Z^n$; then we say that $f$ is of Latt\'es-type,
\item $h$ is of Zorich-type if the translation subgroup $T$ of $G'$ is isomorphic to $\Z^{n-1}$ and $G'$ does not contain a rotation switching the ends of a fundamental beam for the action of $T$ on $\R^n$; then we say that $f$ is of power-type,
\item $h$ is of sine-type if the translation subgroup $T$ of $G'$ is isomorphic to $\Z^{n-1}$ and $G'$ does contain a rotation switching the ends of a fundamental beam for the action of $T$ on $\R^n$; then we say that $f$ is of Chebyshev-type,
\end{enumerate}
\end{definition}

Mayer constructed specific examples of all three types in \cite{Mayer2,Mayer1}. We note that our nomenclature of Latt\`es-type differs slightly from that of Mayer who called all three types here Latt\`es-type. We wish to distinguish between the cases.

By \cite[Theorem 8.3]{MS}, the limits in a fundamental beam for a sine-type or Zorich-type map exist. It is worth pointing out both that the hypotheses of this result apply to strongly automorphic mappings (since $G/T$ is a finite group) and that \cite[Theorem 8.3]{MS} holds when the image is $\overline{\R^n}$ instead of $\R^n$ almost verbatim. Then by \cite[Theorem 8.2]{MS}, the omitted values of a sine-type or Zorich-type map are contained in the set of asymptotic values in a beam. The following lemma does not seem to be stated in \cite{MS}, so we include a proof.

\begin{lemma}
\label{lem:asympomit}
Using the notation above, suppose $h$ is either a sine-type or Zorich-type map with automorphism group $G$ a discrete group of isometries. Moreover, suppose that $a_1 \in \overline{\R^n}$ is an asymptotic value of $h$ in a fundamental beam. Then $h$ omits $a_1$.
\end{lemma}

\begin{proof}
Let $B$ be a fundamental beam for the action of $G$ on $\R^n$. By the hypotheses, there is a sequence $(x_m) \in B$ with $n$'th component $(x_m)_n$ satisfying $|(x_m)_n| \to \infty$ and $h(x_m) \to a_1$.

Suppose for a contradiction that there exists $y\in B$ with $h(y) = a_1$. Then by considering normal neighbourhoods, we can find small neighbourhoods $U,V$ of $y,a_1$ respectively with $h(U) = V$. Let $U' \subset B$ be the corresponding neighbourhood modulo the action of $G$, so $h(U') = V$. Note that the $n$'th coordinate of all elements of $U'$ is bounded above by some constant $C$. However, for all large enough $m$, we find that $h(x_m) = h(y_m)$ for some $y_m \in U'$. Since $x_m \notin U'$ for large enough $m$, we contradict the fact that $h$ is injective on $B$.
\end{proof}

Consequently, a sine-type map omits one value in $\overline{\R^n}$ and a Zorich-type map omits two values in $\overline{\R^n}$.
By choosing a M\"obius map $M$ appropriately, $M\circ h_1$ is still strongly automorphic by Lemma \ref{lem:comp} (i) and the omitted values of $M\circ h_1$ are contained in $\{ 0 ,\infty \}$. 
In what follows, we will assume that the asymptotic values are contained in $\{0,\infty \}$ and then, if necessary, post-compose by a M\"obius map to return to the original strongly automorphic map.

In the sine-type case, $G'|_{\R^{n-1}}$ does not just consist of translations and rotations of $\R^{n-1}$, but also reflections arising from the restriction to $\R^{n-1}$ of rotations of $\R^n$ switching the ends of fundamental beams. 
This corresponds to $z\mapsto -z$ being an invariance for $\cos(z)$, and the restriction to $\R$ given by $x\mapsto -x$ is a reflection.
Denote by $G_{or}$ the subgroup of $G'$ preserving the ends of the fundamental beams or, equivalently, the subgroup of $G'$ which, when restricted to $\R^{n-1}$, consists of orientation preserving maps. Then a fundamental domain for the action of $G_{or}$ is a beam, say $B$.

\begin{lemma}
\label{lem:switch}
With the notation above, there is a unique element $R$ of $G'$ which maps $B$ to itself and switches the ends of $B$.
\end{lemma}

\begin{proof}
Since we are considering sine-type maps, there exists at least one such element.
Suppose there are two rotations $R_1,R_2 \in G'$ with the required properties. Then $R_1\circ R_2$ and $R_2 \circ R_1$ must both be elements of $G_r$ which map $B$ to itself. Consequently, 
\[ R_1 \circ R_2 = R_2 \circ R_1 = Id.\]
Now, for $i=1,2$, $R_i|_{x_n = 0}$ acts as a reflection on $\R^{n-1}$. Denote by $E_i\subset \R^{n-1}$ the fixed point set of $R_i$. These are codimension one hyperplanes in $\R^{n-1}$. Since $R_1$ and $R_2$ commute, standard geometry arguments imply that  $E_1$ and $E_2$ are either perpendicular or are equal. In the first case, $R_1 \circ R_2$ is a non-trivial rotation acting on $\R^n$. However, $R_1\circ R_2$ is the identity and hence $E_1,E_2$ must be equal. It follows that $R_1 = R_2$.
\end{proof}

The upshot of this lemma is that once we restrict to a fixed fundamental beam for the action of $G_{or}$, we can talk about {\it the} rotation acting on it which switches the ends of the beam.

\subsection{Classifying the crystallographic groups}

\begin{theorem}
\label{thm:crystalclass}
A quasiregular map $h:\R^n \to \overline{\R^n}$ is strongly automorphic with respect to a tame quasiconformal group $G$ only if the crystallographic orbifold is spherical.
\end{theorem}

To simplify the notation and since quasiconformal maps are homeomorphisms, we may assume that $G$ itself is a discrete group of isometries.

To clarify the statement of the theorem, in the $\wp$-type case the statement is that $\R^n / G$ is homeomorphic to $S^n$. In the Zorich-type case, the statement is that if $B = X_{n-1} \times \R$ is a fundamental beam for the action of $G$ on $\R^n$, then $h( X_{n-1} \times \{t\} )$ is homeomorphic to $S^{n-1}$ for all $t\in \R$. In the sine-type case, the statement is that if $Y_{n-1} \times (0,\infty)$ is the interior of a fundamental beam for the action of $G$ on $\R^n$, then $h(Y_{n-1} \times \{t\})$ is homeomorphic to $S^{n-1}$ for all $t>0$.

\begin{proof}[Proof of Theorem \ref{thm:crystalclass}]
We prove this theorem in each case separately. First, consider the $\wp$-type case. Since $h:\R^n \to \overline{\R^n}$ is surjective and the closure of the fundamental domain for $h$ is an $n$-polytope, the group $G$ must have $\R^n / G$ homeomorphic to $S^{n}$.

Next, consider the Zorich-type case. We may assume that $0$ and infinity are the omitted values and that the periods of the translation subgroup $T$ of $G$ are perpendicular to $e_n$.
The fundamental set for $\R^n / G$ can be chosen to be a beam, say $B = X_{n-1} \times \R$, where $\overline{X_{n-1}}$ is an $(n-1)$-polytope. For convenience, set $X:= X_{n-1} \times \{ 0 \}$. The crystallographic orbifold in this case is obtained by gluing the faces of $\overline{X}$ according to those that are identified under the action of $G$. In particular, for $t\in \R$, $h$ is a homeomorphism from $X_{n-1} \times \{ t \}$ with the quotient topology onto $h(X_{n-1} \times \{ t \}) \subset \R^n \setminus \{ 0 \}$ with the subspace topology, see Figure \ref{fig:2}.
Therefore there is a homeomorphism that maps $h(x,0) \in h(X)$ to $x\in X$.
It then follows that the composition
\[ ( h(x,0) , t) \mapsto (x,t) \mapsto h(x,t) \]
is a homeomorphism from $h(X) \times \R$ onto $h(X\times \R)$. Since $h(X \times \R) = \R^n \setminus \{ 0 \}$ and $\R^n \setminus \{ 0 \}$ is homeomorphic to $S^{n-1} \times \R$, it follows that $h(X) \times \R$ and $S^{n-1} \times \R$ are homeomorphic. From this and the fact that $M\times\R$ and $M$ are homotopy equivalent for any manifold $M$, we see that $h(X)$ and $S^{n-1}$ are homotopy equivalent and therefore by the (generalized) Poincar\'e conjecture we conclude that $h(X)$ and $S^{n-1}$ are homeomorphic. Applying this argument to $X_t = X_{n-1} \times \{ t \}$ for $t\in \R$ shows $h(X_t)$ is homeomorphic to $S^{n-1}$.

\begin{figure}[h]
\begin{center}
\includegraphics[width=5in]{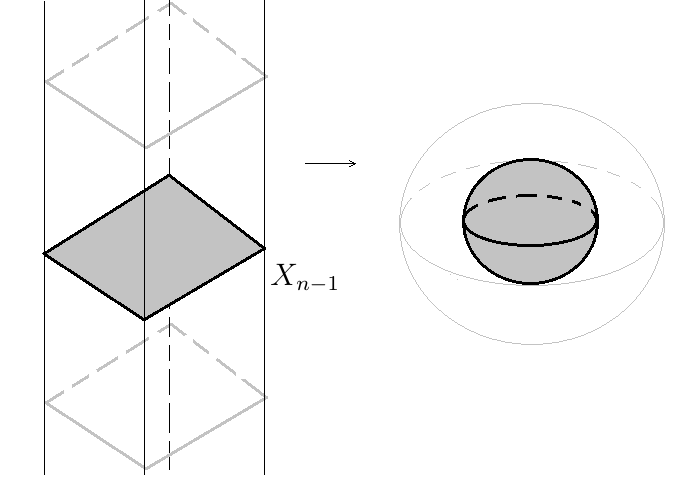}
\caption{Identifying $X_{n-1} \times \R$ and $S^{n-1} \times \R$. The image of $X_{n-1} \times \{ 0 \}$ is a topological  $(n-1)$-sphere.}
\label{fig:2}
\end{center}
\end{figure}

Finally, consider the sine-type case. We may again assume that the omitted value is infinity and that the periods of the translation subgroup $T$ of $G$ are perpendicular to $e_n$.
In this case, a fundamental set $B'$ for $\R^n /G$ can be chosen to be a half-beam. The intersection of $B'$ and the set $\{ x : x_n >0\}$ is $Y_{n-1} \times (0,\infty)$, where $\overline{Y_{n-1}}$ is an $(n-1)$-polytope. Denote by $Y_0$ the intersection of $B'$ and $\{x : x_n = 0\}$. Then $Y_0$ is a subset of $Y_{n-1} \times \{ 0 \}$ that is also an $(n-1)$-polytope since the rotation identifying prime ends of the full beam acts by gluing two halves of the base of $B'$ together. See Figure \ref{fig:3}.

Let $\Omega = \overline{\R^n} \setminus h(Y_0)$. Since $h(Y_0)$ is closed, $\Omega$ is open. Moreover, $\Omega$ is contractible since the inverse image of $\Omega$ under $h^{-1}$ is contractible in $B'$ with the one point compactification at the prime end. The last property to note is that $\Omega$ is simply connected at infinity. This means that for every compact set $E \subset \Omega$, there is a compact set $F \subset \Omega$ containing $E$ so that the induced map $\pi_1(\Omega \setminus F) \to \pi_1(\Omega \setminus E)$ is the zero map. This again follows by considering the inverse images in the one point compactification of $B'$.

Every open contractible subset of $\R^n$ which is simply connected at infinity is known to be homeomorphic to $\R^n$. This is standard for $n=2$, due to Stallings \cite{Stallings} for $n>4$, Freedman \cite{Freedman} for $n=4$ and the original reference for $n=3$ seems hard to locate, although it is certainly contained in \cite{BT}. We therefore have that $\Omega$ is homeomorphic to $\R^n$ and hence $\R^n \setminus h(Y_0)$ is homeomorphic to $\R^n \setminus \{ 0 \}$ which in turn is homeomorphic to $S^{n-1} \times \R$. 

For convenience, set $Y = Y_{n-1} \times \{1\}$. Then there is a homeomorphism from $Y$ with the quotient topology to $h(Y) \subset \R^n \setminus h(Y_0)$. Hence for $y\in Y$ and $t\in (0,\infty)$, the composition
\[ ( h(y,1) , t) \mapsto (y,t) \mapsto h(y,t) \]
is a homeomorphism from $h(Y) \times (0,\infty)$ to $h(Y\times (0,\infty ))$. Since $h(Y) \times (0,\infty)$ is homeomorphic to $h(Y) \times \R$ and $h(Y\times (0,\infty)) = \R^n \setminus h(Y_0)$ is homeomorphic to $S^{n-1} \times \R$, the Poincar\'e conjecture argument again shows $h(Y)$ and $S^{n-1}$ are homeomorphic. Applying this argument to $Y_t = Y_{n-1} \times \{ t \}$ for $t>0$ shows $h(Y_t)$ is homeomorphic to $S^{n-1}$.

\begin{figure}[h]
\begin{center}
\includegraphics[width=5in]{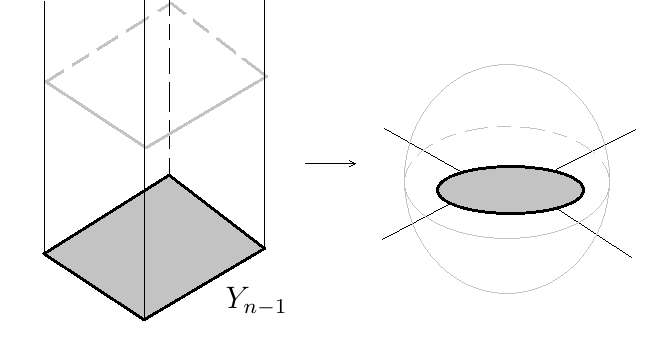}
\caption{Identifying $Y_{n-1} \times (0,\infty)$ and $S^{n-1} \times \R$. The image of the base of $B'$ is a closed topological $(n-1)$-disk.}
\label{fig:3}
\end{center}
\end{figure}

\end{proof}

In the sequel, we will need specific examples of Zorich-type and sine-type mappings which are strongly automorphic with respect to a given crystallographic group with spherical orbifold. We remark that there are various constructions of strongly automorphic mappings in the literature, see \cite{Berg0,BE,Drasin,MS,Mayer2,Mayer1}, although there has been no systematic construction for each possible group that can occur. We leave it as a challenge to the reader to show that every crystallographic group with spherical orbifold in $\R^n$ has a quasiregular map of $\wp$-type which is strongly automorphic with respect to it (see \cite{MS} for an example).

\begin{proposition}
\label{prop:zor}
Let $G$ be a crystallographic group acting on $\R^{n-1}$ with spherical orbifold and translation part $T$ generated by linearly independent vectors $\{w_1,\ldots,w_{n-1}\}$. Then there exists a quasiregular map $Z_G:\R^n \to \R^n \setminus \{ 0 \}$ of Zorich-type that is strongly automorphic with respect to $G$ and so that $Z_G( \{x_n=0 \})$ is the unit sphere in $\R^n$.
\end{proposition}

\begin{proof}
By applying a preliminary orthogonal mapping if necessary, we may assume that the span of $\{w_1,\ldots, w_{n-1} \}$ agrees with the span of $\{e_1,\ldots, e_{n-1} \}$ in $\R^n$.
Let $B$ be a fundamental set for the action of $G$ on $\R^n$. We may assume $B$ is connected and in fact a beam $X\times \R$, where $\overline{X}$ is an $(n-1)$-polytope. The group $G$ acts on $B$ by gluing $(n-2)$-dimensional faces of $\overline{X}$ together.

With this identification, we may triangulate $\overline{X}$ in such a way that $X$ is PL-homeomorphic to a PL $(n-1)$-sphere $\Sigma_{n-1}$. The standard orientation that $\overline{X} \times \{t\}$ inherits from $\R^{n-1}$ induces an orientation on $\Sigma_{n-1}$.
We can then find an embedding $\iota: \Sigma_{n-1} \to S^{n-1}$ so that on each face of $\Sigma_{n-1}$, $\iota$ is smooth. 
We obtain a map $Z_G : \{x_n = 0\} \to S^{n-1}$ that by construction is strongly automorphic with respect to $G$. 
This map is BLD since it arises as a composition of a PL-homeomorphism and $\iota$, which is smooth almost everywhere. Moreover, it is orientation preserving as a map $\R^{n-1} \to S^{n-1}$ and hence is quasiregular.

To extend the domain of definition of $Z_G$, we use the formula
\[ Z_G(x_1,\ldots, x_{n-1},x_n) = e^{\pm x_n}Z_G(x_1,\ldots, x_{n-1},0)\]
in $\overline{B}$, where the sign is chosen so that the extended map is also orientation preserving.
Finally, we extend to all of $\R^n$ via the group $G$ to obtain $Z_G$. If we require $\lim_{x\in B, x_n \to \infty}Z_G(x) = +\infty$, and the choice of sign does not give this, we may post-compose by a M\"obius map to obtain this.
See Figure \ref{fig:4}.

\begin{figure}[h]
\begin{center}
\includegraphics[width=5in]{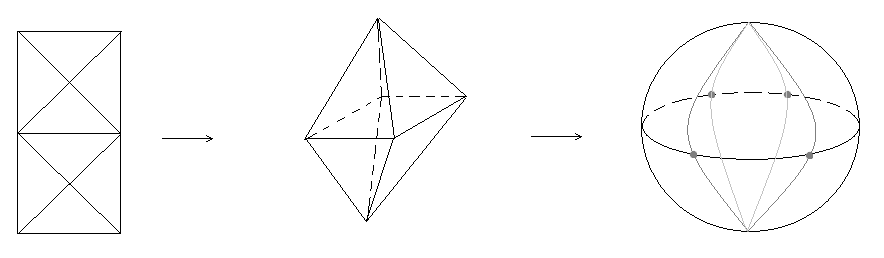}
\caption{Mapping $\overline{X}$ to a sphere.}
\label{fig:4}
\end{center}
\end{figure}

To check that $Z_G$ is indeed quasiregular is a computation analogous to the original computation by Zorich \cite{Zorich}, see also \cite{Berg0,IM,Rickman}.
\end{proof}

\begin{proposition}
\label{prop:sine}
Let $G$ be a crystallographic group acting on $\R^{n-1}$ with spherical orbifold and translation part $T$ generated by linearly independent vectors $\{w_1,\ldots,w_{n-1}\}$.  Then there exists a quasiregular map $S_G:\R^n \to \R^n$ of sine-type that is strongly automorphic with respect to $G_1 = <G,R>$, where $R$ is a rotation 
mapping a fundamental beam for the action of $G$ on $\R^n$ into itself and switching the two prime ends.
Moreover, 
\[S( \{x_n = 0\}) = D_{n-1}:= \{x : x_1^2 + \ldots + x_{n-1}^2 \leq 1, x_n = 0\}.\]
\end{proposition}

\begin{proof}
We may again assume that the span of $\{w_1,\ldots, w_{n-1} \}$ agrees with the span of $\{e_1,\ldots, e_{n-1} \}$ in $\R^n$.
Let $B$ be a fundamental set for the action of $G$ on $\R^n$. We may assume that $B$ is a beam. The action of $R$ on $B$ yields a fundamental domain $B'$ that we may assume is a beam $Y\times (0,\infty)$.
As in the previous proposition, $\overline{Y}$ is an $(n-1)$-polytope.

Via a suitable triangulation of $\overline{Y} \times [0,1]$, we can find a map $h: \overline{Y} \times [0,1] \to \{x \in \R^n : x_n \geq 0, |x| \leq e \}$ which maps the set $\overline{Y} \times \{ 1 \}$ onto $\{ x\in \R^n : x_n\geq 0, |x| = e \}$
and the set $\overline{Y} \times \{ 0 \}$ onto $D_{n-1}$. The map $h$ is to be chosen so that it is smooth on each face of the triangulation.

From here, we follow the construction outlined in \cite{BE} (see also \cite{Drasin}).
We extend to $\overline{B} \cap \{x_n \geq 1\}$ as in the previous proposition via
\[ h(y,t) = e^{t-1}h(y,1).\]
We can then reflect in the $x_n = 0$ hyperplane in the domain and the range to extend $h$ to $\overline{B'}$. Finally, we extend to all of $\R^n$ via the group $G_1$ to obtain the required quasiregular map $S_G$. The computations that $S_G$ is quasiregular are omitted, see \cite{BE,Zorich}.
\end{proof}

\section{Strongly automorphic linearizers}

In this section, with the constructions of the previous section in hand, we will prove Theorem \ref{thm:1}, relating solutions of a Schr\"oder equation and linearizers, and Corollary \ref{cor:1}, on the corresponding Julia sets. 
First let $G$ be a tame quasiconformal group acting on $\R^n$.

\begin{lemma}
\label{lem:winding}
Let $\Stab (0) \subset G$ be the subgroup of $G$ which fixes $0$. Then there is a quasiregular map $W:\R^n \to \R^n$ which fixes $0$, is strongly automorphic with respect to $\Stab(0)$ and so that if $A$ is a loxodromic repelling uqc map satisfying $AGA^{-1} \subset G$, there is a loxodromic repelling uqc map $A_1$ satisfying $W \circ A = A_1 \circ W$.
\end{lemma}

\begin{proof}
We first pass to the group $G' = \varphi^{-1}G\varphi$ and denote by $G'_0$ the stabilizer of $0$ in $G'$.

The group $G'_0$ is a finite subgroup of $O(n)$ consisting only of rotations. We can thus view $G'_0$ as acting on $S^{n-1}$. We realize the action of $G'_0$ on $S^{n-1}$ in the PL setting. That is, let $X\subset \R^n$ be an $(n-1)$-polytope so that there is an embedding $\iota_1 : X \to S^{n-1}$ which is smooth on the faces of $X$ and, moreover, $G'_0$ acts on $X$ too. For clarity, denote by $H$ the action of $G'_0$ on $X$. Then for any $h\in H$, there exists $g\in G'_0$ so that $g\circ \iota_1 = \iota_1 \circ h$.

By the main result of \cite{Lange} and including the point at infinity, $S^{n-1} / G_0'$ is PL-homeomorphic to $S^{n-1}$.
This means that there is a refinement $X'$ of $X$ and a PL map $\alpha:X' \to \Sigma_{n-1}$, where $\Sigma_{n-1}$ is a PL $(n-1)$-sphere. If $D$ is a fundamental set for the action of $H$ on $X$, $H$ acts by gluing faces of $\overline{D}$ together in such a way that we obtain $\Sigma_{n-1}$. We can thus pull the triangulation giving $\Sigma_{n-1}$ to $\overline{D}$, and then to the rest of $X$ via $G$. This construction yields $X'$, and immediately implies that $\alpha$ is strongly automorphic with respect to $H$.

Since we can find an embedding $\iota_2:\Sigma_{n-1} \to S^{n-1}$, we obtain a map $W_1:S^{n-1} \to S^{n-1}$ satisfying $\iota_2\circ \alpha = W_1\circ \iota_1$. By construction, this map is BLD since $\alpha$ is PL and the embeddings are smooth on the faces. Hence we can extend $W_1$ to all of $\R^n$ via radial extension, that is, in spherical coordinates $W_1(\sigma,t) = tW_1(\sigma,1)$ for $\sigma \in S^{n-1}$ and $t\geq 0$. This extension is also BLD and hence quasiregular. The fact $W_1$ is strongly automorphic with respect to $G'_0$ follows from the fact that $\alpha$ is strongly automorphic with respect to $H$.

Then $W = \varphi \circ W_1 \circ \varphi^{-1}$ is our required map, and is clearly strongly automorphic with respect to $\Stab(0)$. If $U$ is a fundamental domain for $G_0'$, then $\varphi(U)$ is a fundamental domain for $\Stab(0)$. Hence if $A$ satisfy the hypotheses of the lemma, then for any choice of a branch of the inverse of $W^{-1}$, the map $A_1 := W\circ A\circ W^{-1}$ is well-defined. Moreover, $A_1$ is bijective and loxodromic repelling uniformly quasiconformal since $A$ is. This completes the proof.
\end{proof}

\begin{proof}[Proof of Theorem \ref{thm:1}]

First suppose that $f$ is uqr with repelling fixed point $x_0$, generalized derivative $\psi$ at $x_0$ and corresponding linearizer $L$, which is also strongly automorphic with respect to the tame quasiconformal group $G$, and so that $f\circ L = L\circ \psi$.
If $g\in G$, we have
\[ (L\circ \psi) \circ g = ( f\circ L)\circ g = f\circ L = L\circ \psi = L \circ g \circ \psi\]
by using the first condition in the strong automorphy of $L$ and the linearizer equation. 
Therefore $L \circ (\psi \circ g \circ \psi^{-1}) = L$.
Then by the second condition in the strong automorphy of $L$, we see that
\[ \psi \circ g \circ \psi^{-1} \in G\]
for each $g\in G$. Since $\psi (0)=0$, we see that $\psi$ satisfies the conditions for $A$ given in Theorem \ref{thm:schroder}. Hence there is a unique solution $f_1$ to the Schr\"oder equation $f_1 \circ L = L \circ\psi$ of either power-type, Chebyshev-type or Latt\`es type. However, $f$ satisfies this equation and so $f_1=f$. 

For the converse direction suppose that a solution $f$ of the Schr\"oder equation $f\circ h = h\circ A$ is either of power-type, Chebyshev-type or Latt\`es-type. Suppose we find a linearizer $L$ of $f$ at $L(0)$ satisfying $f\circ L = L \circ \psi$ for some generalized derivative $\psi$ of $f$ at $L(0)$.

First assume that $0$ is not in the branch set of $h$. By \cite[Theorem 1.2]{HM} every loxodromic repelling uniformly quasiconformal map is quasiconformally conjugate to $x\mapsto 2x$. Applying this to $A$ and $\psi$, there is a quasiconformal map $\alpha$ so that $A = \alpha^{-1} \circ \psi \circ \alpha$. Hence
\[ f\circ h = h \circ \alpha ^{-1} \circ \psi \circ \alpha,\]
and so
\[ f\circ (h \circ \alpha^{-1}) = (h \circ \alpha^{-1}) \circ \psi.\]
Since $0$ is not in the branch set of $h \circ \alpha^{-1}$, then $h \circ \alpha^{-1}$ is a linearizer for $f$ at $h(0)$. Now since $h$ is strongly automorphic, it follows by Lemma \ref{lem:comp} that $h  \circ \alpha^{-1}$ is strongly automorphic. Next, by \cite[Corollary 3.7]{FM}, if one element of the set of linearizers at a given point is strongly automorphic then they all are. Hence $L$ is strongly automorphic.

Finally, assume that $0$ is in the branch set of $h$. Apply Lemma \ref{lem:winding} to find a quasiregular map $W$ which is strongly automorphic with respect to $\Stab(0)$. Since near $0$, $h$ and $W$ are both strongly automorphic with respect to $\Stab(0)$, it follows that $h \circ W^{-1}$ is a well-defined quasiregular map which is locally injective near $0$. By Lemma \ref{lem:winding} and the Schr\"oder equation, we have
\begin{align*}
f \circ (h\circ W^{-1}) &= (f\circ h) \circ W^{-1} \\
&= (h \circ A) \circ W^{-1} \\
&= h\circ W^{-1} \circ A_1.
\end{align*}
As in the case above, we can conjugate $A_1 $ to $\psi$ via $\alpha$ and obtain that $h\circ W^{-1} \circ \alpha^{-1}$ is a linearizer for $f$ at $h(0)$. By  \cite[Theorem 3.1]{FM}, we have $L = h\circ W^{-1} \circ \xi$ for some quasiconformal map $\xi$. Hence if we set $P = \xi^{-1} \circ W$, we obtain that $L \circ P$ is strongly automorphic.

\end{proof}

\begin{proof}[Proof of Corollary \ref{cor:1}]
Suppose $L$ is strongly automorphic with respect to $G$, $A$ is a loxodromically repelling uqc map such that $AGA^{-1} \subset G$ and $f$ is the corresponding solution to the Schr\"oder equation $f\circ L = L\circ A$. 
We again have three cases.

{\bf Case 1: $\wp$-type.} In this case, by Theorem \ref{thm:1}, $f$ is of Latt\`es-type. If $U \subset \overline{\R^n}$ is any open set, then for any $m\geq 1$, we have $f^m(L(U)) = L(A^m(U))$. Since $L$ is of $\wp$-type and $A$ is loxodromically repelling, it follows that for $m$ large enough, $f^m(L(U)) = \overline{\R^n}$. Consequently every point of $\overline{\R^n}$ has the blowing-up property and so $J(f) = \overline{\R^n}$.

{\bf Case 2: Zorich-type.}
In this case, $L$ is strongly automorphic with respect to $G = \varphi G' \varphi^{-1}$, where $G'$ is a discrete group of isometries, $\varphi$ is a quasiconformal mapping and the translation subgroup $T$ of $G'$ satisfies $T \cong \Z^{n-1}$. 
Further, $L$ omits two values that we may assume are $0$ and infinity. 

Suppose that $T$ is generated by translations with respect to the linearly independent set $\{ w_1,\ldots, w_{n-1} \}$. Let $\alpha$ be a rotation so that the hyperplane spanned by $\{ \alpha(w_1), \ldots, \alpha(w_{n-1}) \}$ agrees with the hyperplane spanned by $\{e_1,\ldots, e_{n-1}\}$.
Then by Lemma \ref{lem:comp}, $L_1 = (\varphi \circ \alpha)^{-1} \circ L \circ \varphi \circ \alpha$ is strongly automorphic with respect to a group $G_1$ whose translation subgroup $E$ is generated by $\{ \alpha(w_1), \ldots, \alpha(w_{n-1}) \}$. Denote by $B_1$ a fundamental set for $G_1$ given by
$B_1 = X \times \R$ where $\overline{X}$ is an $(n-1)$-polytope.

Recall the Zorich map $Z_{G_1}$ constructed in Proposition \ref{prop:zor}, applied to $\{\alpha(w_1), \ldots, \alpha (w_{n-1}) \}$ and $G_1$ which has the same fundamental beam as $L_1$.
Both $Z_{G_1}$ and $L_1$ are injective on the union of the interior of $B_1$ with some of the boundary, including points identified under $G_1$ only once. The image in both cases is $\R^n \setminus  \{ 0 \}$.
Hence we can define a quasiconformal map $g:\R \setminus \{ 0 \} \to \R \setminus \{ 0 \}$ by $g = L_1 \circ Z_{G_1}^{-1}$, where we choose the branch of the inverse of $Z_{G_1}$ with image $B_1$ as indicated above. Clearly $\lim _{x\to 0} g(x) = 0$ and $\lim_{x\to \infty} g(x) = \infty$ and so we can extend $g$ to a quasiconformal map on all of $\overline{\R^n}$.

By construction, $Z_{G_1}(\{x:x_n=0\} ) = S^{n-1}$ and so $L_1(\{x:x_n = 0\}) = g(S^{n-1})$. Now, since $f\circ L = L\circ A$ it follows that if we write $\beta = \varphi \circ \alpha$ and $A_1 = \beta^{-1} \circ A \circ \beta$, then
\begin{equation}
\label{eq:ztype} 
(\beta^{-1} \circ f \circ \beta) \circ L_1 = L_1 \circ A_1.
\end{equation}
It is not hard to check that $A_1 G_1 A_1^{-1} \subset G_1$ and hence $f_1 := \beta^{-1} \circ f \circ \beta$ is the unique solution of the Schr\"oder equation \eqref{eq:ztype}. Since $A$ is loxodromically repelling and uniformly quasiconformal, it follows that $A_1$ is too. Moreover, since $A_1G_1 A_1^{-1} \subset G_1$, it follows that if $x \in \R^n$ with $x_n >0$ then $(A_1^m(x))_n \to \infty$ and if $x\in \R^n$ with $x_n <0$ then $(A_1^m(x))_n \to - \infty$.
Hence if $x\in \R^n$ with $x_n \neq 0$, then $f_1^m(L_1(x))$ converges to one of the two superattracting fixed points of $f_1$, at $0$ and $\infty$. The Julia set of $f_1$ is therefore equal to $L_1(\{x:x_n = 0\}) = g(S^{n-1})$. Since $f$ is a quasiconformal conjugate of $f_1$, it follows that $J(f) = \beta(g(S^{n-1}))$ which is a quasisphere.

{\bf Case 3: sine-type.}
The idea for the sine-type case is similar to the Zorich-type case, but we will include the details for the convenience of the reader.
As above, $L$ is strongly automorphic with respect to $G=\varphi G' \varphi^{-1}$ and $f$ is a Chebyshev-type uqr mapping satisfying $f\circ L = L\circ A$ with $A G A^{-1} \subset G$. This time the only omitted value of $L$ is infinity.

Again, find a rotation $\alpha$ so that $L_2  = (\varphi \circ \alpha)^{-1} \circ L \circ \varphi \circ \alpha$ is strongly automorphic with respect to a group $G_2$ which has translation subgroup $E$ generated by $\{ \alpha(w_1),\ldots,\alpha(w_{n-1}) \}$. Then $L_2$ has the same fundamental half-beam $B_2$ as the sine-type mapping $S_{H_2}$ constructed in Proposition \ref{prop:sine}, where $H_2 = G_2 / R$ and $R$ is the rotation in $G_2$ identifying ends of the fundamental beam, recalling Lemma \ref{lem:switch}.

As before, we can construct a quasiconformal map $g:\R^n \to \R^n$ via the equation $g = L_2\circ S_{H_2}^{-1}$, where we choose the branch of the inverse of $S_{H_2}$ with image $B_2$ as indicated above. 
Clearly $\lim_{x\to \infty} g(x) = \infty$ and so $g$ extends to a quasiconformal mapping of $\overline{\R^n}$.

By construction, $S_{H_2}( \{x : x_n = 0\}) = D_{n-1}$ and so $L_2( \{x:x_n = 0\}) = g(D_{n-1})$. Exactly as in the Zorich-type case, we can solve the Schr\"oder equation $f_2 \circ L_2 = L_2 \circ A_2$ with $f_2 = \beta^{-1} \circ f \circ \beta$. Since $J(f_2) = g(D_{n-1})$ we conclude that $J(f) = \beta(g(D_{n-1}))$ and so the Julia set of $f$ is a quasi-disk.

These three cases complete the proof of Corollary \ref{cor:1}.
\end{proof}

\section{A quasiregular analogue of $(z+1/z)/2$}

In this section, we will use the constructions above to give a quasiregular version of the rational map 
 $P(z) =  (z + 1/z)/2$.  We recall that one way to define a Chebyshev polynomial $T_d$ of degree $d$ is via the equation 
\[T_d(P(z)) = P(z^d).\]
We can therefore aim to generalize the construction of $P$ through the quasiregular mappings of power-type and Chebyshev-type.

Let $G$ be a crystallographic group with spherical orbifold, acting on $\R^{n-1}$ for $n\geq 2$. Let $G'$ be the group generated by $G$, viewed as acting on $\R^n$, and the rotation $R$ identifying prime ends of a fundamental beam for $G$.
Recall the Zorich-type map $Z_G$ and sine-type map $S_{G}$ from the proofs of Propositions \ref{prop:zor} and \ref{prop:sine} respectively. These have fundamental sets of a beam $B$ and half-beam $B^+$ respectively, where the half-beam is obtained by quotienting $B$ via $R$.

On $\R^n \setminus \{ 0 \}$, choose a branch of $Z_G^{-1}$ with image $B$. Then $S_{G}: B \to \R^n$ is a surjective two-to-one mapping. Consequently, if we define 
\begin{equation}
\label{eq:h}
h_1 = S_{G} \circ Z_G^{-1}, 
\end{equation}
we obtain a two-to-one map from $\R^n \setminus \{ 0 \}$ onto $\R^n$. It is not hard to see that $h_1$ must have poles at $x=0$ and at infinity. Hence $h_1:\overline{\R^n} \to \overline{\R^n}$ is a degree two mapping.

With the same branch of the inverse of $Z_G$ as above, let $I:\R^n \to \R^n$ be given by 
\begin{equation}
\label{eq:I}
I = Z_G \circ R \circ Z_G^{-1}.
\end{equation} 
Then $I$ is an analogue of $1/z$ in dimension two and switches the two components of $\overline{\R^n} \setminus S^{n-1}$. Since $S_{G} \circ R = S_{G}$, we have
\[ h_1 \circ I = S_{G} \circ Z_G^{-1} \circ ( Z_G \circ R \circ Z_G^{-1} ) = S_{G} \circ R \circ Z_G^{-1} = S_{G}\circ Z_G^{-1} = h_1.\]
This construction is the main idea behind the proof of Theorem \ref{thm:zz}.

\begin{proof}[Proof of Theorem \ref{thm:zz}]
Since $U$ is an $(n-1)$-quasisphere and $V$ is an $(n-1)$-quasidisk, we can find quasiconformal maps $A,B$ from $\overline{\R^n}$ onto itself so that $A(U) = S^{n-1}$ and $B(V) = D_{n-1}$. Set $h = B^{-1} \circ h_1 \circ A$, where $h_1$ is defined by \eqref{eq:h}. Then $h$ is the required degree two map and moreover, it satisfies $h\circ \rho = h$, where $\rho = A^{-1} \circ I \circ A$ and $I$ is given by \eqref{eq:I}. To see this, we have
\begin{align*} 
h \circ \rho &= h\circ A^{-1} \circ I \circ A\\
& = B^{-1} \circ h_1 \circ A \circ A^{-1} \circ I \circ A \\
&=B^{-1} \circ h_1 \circ I \circ A \\
&= B^{-1} \circ h_1 \circ A\\
&= h.
\end{align*}
Clearly $\rho$ switches the components of $\overline{\R^n} \setminus U$, and this completes the proof.
\end{proof}

It is worth remarking that modifying $h_1$ by inserting a dilation changes the degree but not the distortion. More precisely, if $d\in \N$, then as above construct the map $h_d(x) = S_{G}(dZ_G^{-1}(x))$. This is a quasiregular analogue of $(z^d + 1/z^d)/2$, and by choosing $d$ large enough we can guarantee the degree is larger than the distortion. Note that the degree of $h_d$ is $2d^{n-1}$. We leave the study of the dynamics of this map for future work.

We end this section by computing the branch set of $h_1$.

\begin{proposition}
\label{prop:branch}
The branch set of $h_1$ consists of $(n-2)$-dimensional subsets of $S^{n-1}$ and various hyperplanes.
\end{proposition}

\begin{proof}
One can check that the branch set of $S_G$ consists of edges of the beams forming the closures of fundamental domains for the action of $G'$ on $\R^n$, recalling $G' = <G,R>$, together with the intersection of the boundaries of these beams with the hyperplane $x_n = 0$. The branch set of $h_1$ is then given by $Z_G(B_{S_G})$. The edges of the beams are mapped to $(n-2)$-dimensional subsets of hyperplanes joining $0$ to infinity by $Z_G$ and the intersection of the boundaries of the beams with the plane $x_n=0$ are mapped into $S^{n-1}$ by $Z_G$. 

\end{proof}

\section{Conical points}

We now turn to the converse of Corollary \ref{cor:1}. Suppose the Julia set of a uqr map $f$ is a quasidisk, a quasisphere or all of $\overline{\R^n}$.
We cannot in general classify the maps based on the topology of their Julia sets, but we can, in the sense that the uqr map must agree with a power-type, Chebyshev-type or Latt\`es-type map on its Julia set, if a further condition holds.

We start with a slight restatement of Miniowitz's version of Zalcman's Lemma for quasiregular mappings.

\begin{theorem}[\cite{Miniowitz}, Lemma 1]
Let $f:\overline{\R^n} \to \overline{\R^n}$ be a uqr map. If $x_0$ and $f(x_0)$ are both in $\R^n$, then $x_0\in J(f)$ if and only if there exist sequences $x_j \to x_0$, $k_j \in \N$ with $k_j \to \infty$ and $\alpha_j > 0$, $\alpha_j \to 0$, and a non-constant quasiregular map $\Psi:\B^n \to \R^n$ so that
\[f^{k_j}(x_j + \alpha_jx) \to \Psi(x)\]
uniformly on compact subsets of $\B^n$.  
\end{theorem}

The convention here is that if $x_0$ is a pole of $f$, or $x_0$ is the point at infinity, then we first conjugate $f$ by a M\"obius map $M$ to move both points in domain and range to $\R^n$.
If the sequence $x_j$ can in fact be chosen to be always equal to $x_0$, then we have the following definition.

\begin{definition}
Let $f:G \to \R^n$ be a uqr map defined on a domain $G\subset \R^n$. We say that $f$ has a {\it conical point} $x_0 \in G$ if there exists a non-constant quasiregular map $\Psi : \B^n \to \R^n$, an increasing sequence $k_j\in \N$ and a sequence $\alpha_j \to 0$, $\alpha_j > 0$, so that
\[f^{k_j}(x_0 + \alpha_jx) \to \Psi(x)\]
uniformly on compact subsets of $\B^n$. The set of conical points is denoted by $\Lambda (f)$.
\end{definition}

We can similarly consider conical points when either $x_0$ or $f(x_0)$ are the point at infinity by conjugating by an appropriate M\"obius $M$ map and checking if the condition holds. We remark that both the condition in Miniowitz's result and the condition for a conical point are unchanged at points which are not at infinity or mapped to infinity when conjugating by a M\"obius map $M$.

It is immediate from the definition of the set of conical points, and Miniowitz's version of Zalcman's Lemma, that $\Lambda(f) \subset J(f)$. Moreover, $\Lambda(f)$ is completely invariant under $f$, see \cite[Lemma 3.4]{MM}. For our purposes, we need to know that conical points are invariant under conjugation.

\begin{lemma}
\label{lem:dist} 
Suppose $f:\overline{\R^n} \to \overline{\R^n}$ is uqr, and $g:\overline{\R^n} \to \overline{\R^n}$ is quasiconformal.  
Then $\Lambda(g\circ f \circ g^{-1}) = g(\Lambda(f))$.
\end{lemma}

\begin{proof}
Suppose first that $x_0$ and $f(x_0)$ are in $\R^n$ and that $x_0 \in \Lambda (f)$. Then if $h(x) = x+x_0$ and $f_1 = h^{-1}\circ f \circ h$ it follows that $0\in \Lambda(f_1)$.
It therefore suffices to assume $0\in \Lambda(f)$ and then show that $g(0) \in \Lambda(g\circ f\circ g^{-1})$, assuming $g(0)$ is not the point at infinity.
From the definition, we find a decreasing sequence of positive real numbers $\alpha_j \to 0$, an increasing sequence of positive integers $k_j$ and a non-constant quasiregular map $\Psi :\B^n \to \R^n$ so that
\[f^{k_j}(\alpha_j x) \to \Psi(x)\]
uniformly on $\B^n$. For convenience, let $y_0 = \Psi(0)$. Since $\Psi$ is quasiregular on $\B^n$, for any $r<1$, there exists $R>0$ so that $\Psi(B(0,r)) \subset B(y_0,R)$.
In particular, for $j$ large enough, 
\begin{equation}
\label{eq:dist1} 
f^{k_j}(B(0,\alpha_j/2)) \subset B(y_0,2R).
\end{equation} 

Next, since $g^{-1}$ is quasiconformal, then by Theorem \ref{thm:lindist} there exists $r_0 >0$ and a constant $C$ depending only on $K(g)$ so that
\[ \frac{ L(g(0),r,g^{-1} )}{l(g(0),r,g^{-1} )} \leq C\]
for $r<r_0$. We choose a sequence $\beta_j \to 0$ so that 
\begin{equation}
\label{eq:dist2}
L(g(0),\beta_j,g^{-1} ) = \frac{\alpha_j}{2}. 
\end{equation}
Then for large enough $j$,
\begin{equation}
\label{eq:dist3} 
l(g(0),\beta_j,g^{-1}) \geq \frac{\alpha_j}{2C}.
\end{equation}
For $x\in \B^n$ define
\[ \varphi_j(x) = f^{k_j}(g^{-1}(g(0) + \beta_j x) ),\]
where $k_j$ is the sequence above. 
Then by \eqref{eq:dist1} and \eqref{eq:dist2}, for large enough $j$ we have
\[ \varphi_j(\B^n) \subset f^{k_j}(B(0,\alpha_j/2)) \subset B(y_0,2R).        \]
Consequently, by the quasiregular version of Montel's Theorem the family $\{ \varphi_j |_{\B^n} \}$ is a normal family. Passing to a subsequence if necessary, we see that $\varphi_j \to \Phi$ uniformly on $\B^n$, where $\Phi$ is a quasiregular map. We need to prove that $\Phi$ is non-constant.

To that end, since $\Psi$ is a non-constant quasiregular map, there exists $S>0$ such that $\Psi(B(0,1/2C)) \supset B(y_0,S)$, where $C$ is the constant in \eqref{eq:dist3}. By the quasiregular version of Hurwitz's Theorem, see for example \cite[Lemma 2]{Miniowitz}, for large enough $j$ we have
\[ f^{k_j}(B(0,\alpha_j/2C )) \supset B(y_0,S/2).\]
Therefore, by \eqref{eq:dist3} for large enough $j$ we have
\[ \varphi_j(\B^n) \supset B(y_0,S/2).\]
Again appealing to Hurwitz's Theorem, we conclude that $\Phi (\B^n) \supset B(y_0,S/2)$ and so $\Phi$ is non-constant.

It follows that $(g\circ f \circ g^{-1})^{k_j}(g(0) + \beta_j x)$ converges uniformly on $\B^n$ to the non-constant quasiregular map $g\circ \Phi$. Hence $g(0)$ is a conical point of $g\circ f \circ g^{-1}$. This shows that $g(\Lambda(f)) \subset \Lambda(g\circ f \circ g^{-1})$. Finally, if $y_0 \in \Lambda(g\circ f \circ g^{-1})$, then the argument above shows that $g^{-1}(y_0) \in \Lambda (f)$ and so $\Lambda (g\circ f \circ g^{-1}) = g(\Lambda(f))$.

Finally, if either $g(0)$, $x_0$ or $f(x_0)$ are the point at infinity, we may conjugate by an appropriate M\"obius map so that all points under consideration are in $\R^n$ and then apply the above argument.
\end{proof}

\begin{lemma}
\label{lem:getp}
Suppose $f:\overline{\R^n} \to \overline{\R^n}$ is a uqr mapping with $J(f) = D_{n-1}$.  
Then there exists a uqr map $P:\overline{\R^n} \to \overline{\R^n}$ such that $f\circ h_1 = h_1 \circ P$, where $h_1$ is defined by \eqref{eq:h}. Moreover, $J(P) = S^{n-1}$.
\end{lemma}

\begin{proof}
In $\overline{\R^n} \setminus D_{n-1}$ there are two branches of $h_1^{-1}$. Denote them by $\varphi_1 : \overline{\R^n} \setminus D_{n-1} \to \B^n$ and $\varphi_2 : \overline{\R^n} \setminus D_{n-1} \to \overline{\R^n} \setminus \overline{\B^n}$. Since the involution $\rho$ satisfies $h_1 \circ \rho = h_1$, it follows that $\varphi_1 = \rho \circ \varphi_2$.

Suppose $y_0 \in D_{n-1}$ and $y_m \in \overline{\R^n} \setminus D_{n-1}$ with $y_m \to y_0$. We would like to assign a value to $\varphi_1(y_0)$, but $\varphi_1$ does not extend continuously to $D_{n-1}$. However, we do have the following three cases:
\begin{enumerate}[(i)]
\item If $y_0 \in D_{n-1} \cap h_1(B_{h_1})$, then we can extend $\varphi_1$ continuously to $y_0$. Recall $h_1$ is of degree two, so any branch point of $h_1$ must have local index two and so $y_0$ has a unique pre-image.
\item If $y_0 \in D_{n-1} \setminus h_1(B_{h_1})$ and $y_m \to y_0$ with the $n$'th component of $y_m$ strictly positive for all large enough $m$, then $\varphi_1(y_m)$ converges to a unique point on $\partial \B^n$ that we can informally say is on the upper hemisphere.
\item If $y_0 \in D_{n-1} \setminus h_1(B_{h_1})$ and $y_m \to y_0$ with the $n$'th component of $y_m$ strictly negative for all large enough $m$, then $\varphi_1(y_m)$ converges to a unique point on $\partial \B^n$ that we can informally say is on the lower hemisphere.
\end{enumerate}
In each case, if we are given a sequence $y_m\to y_0$ as above, we can identify a limit of $\varphi_1(y_m)$. We can similarly find limits of $\varphi_2(y_m)$.

Now, if $f$ has $J(f) = D_{n-1}$, we can define the map $P$ on $\overline{\R^n} \setminus \partial \B^n$ as follows:
\[ P = \left \{ \begin{array}{ll} \varphi_1\circ f\circ h_1 :\B^n \to \B^n \\ \varphi_2 \circ f\circ h_1 :\overline{\R^n} \setminus \overline{\B^n} \to \overline{\R^n} \setminus \overline{\B^n} \end{array} \right .\]
and in particular we see that $\rho \circ P = P \circ \rho$ in $\overline{\R^n} \setminus \partial \B^n$.

We extend $P$ to $\partial \B^n$ by continuity. That is, suppose $x_0 \in \partial \B^n$ and $x_m \to x_0$ with $x_m \in \B^n$ for all $m$. Then both $h_1(x_m)$ and $f(h_1(x_m))$ are sequences that must fall into one of the three cases above (it is possible that, given $x_m$, these two sequences might fall into different cases). Hence $\varphi_1(f(h_1(x_m)))$ has a well-defined limit that we denote by $P(x_0)$. Observe that this is independent of the choice of sequence $x_m \to x_0$.

If $z_m \in \overline{\R^n} \setminus \overline{\B^n}$ with $z_m \to x_0$, then $z_m = \rho(z_m')$ for a sequence $z_m' \in \B^n$ which converges to $\rho(x_0)$. Then
\[ P(z_m) =  \varphi_2 (f(h_1(z_m))) = \rho( \varphi_1 ( f(h_1( \rho (z_m'))))) \to \rho(P(\rho(x_0))) = P(x_0).\]
Since this is independent of the choice of $z_m$, we can extend the domain of definition of $P$ to all of $\overline{\R^n}$.

By construction, $P$ satisfies $h_1 \circ P = f \circ h_1$ as required. To see that $P$ is uqr, we have $h_1 \circ P^m = f^m \circ h_1$ and may then use the fact that $f$ is uqr. 
Finally, by construction $\B^n$ and $\overline{\R^n} \setminus \overline{\B^n}$ are completely invariant domains for $P$.
Moreover, if $U$ is any domain intersecting $\partial \B^n$, the semi-conjugacy between $P$ and $f$ implies the forward orbit $O^+(U)$ under $P$ can only possibly omit $0$ and infinity. Consequently $J(P) = \partial \B^n$.
\end{proof}

We can now prove Theorem \ref{thm:2}, that states if the Julia set equals the set of conical points and is either a quasisphere, quasidisk or all of $\overline{\R^n}$, then the uqr map agrees with a power-type map, a Chebyshev-type map or a Latt\`es-type map on its Julia set respectively.

\begin{proof}[Proof of Theorem \ref{thm:2}]
We will deal with the three cases separately.
\begin{enumerate}[(i)]
\item First, if $J(f) = \Lambda(f) = \overline{\R^n}$, for $n\geq 3$, then by \cite[Theorem 1.3]{MM} $f$ is of Latt\`es type.

\item Next, if $J(f) = \Lambda (f)$ is an $(n-1)$-quasisphere for $n\geq 4$, then conjugate $f$ by a quasiconformal map $g$ so that $J(g\circ f \circ g^{-1}) = S^{n-1}$. By Lemma \ref{lem:dist}, $\Lambda(g\circ f \circ g^{-1}) = S^{n-1}$ too. By \cite[Corollary 6.2]{MM}, $\ell := g\circ (f|_{J(f)} )\circ g^{-1}$ restricted to $S^{n-1}$ is a Latt\`es-type map in the sense of \cite{MM}. As observed in \cite{MM}, which references \cite{Mayer1}, Latt\'es-type maps on $S^{n-1}$ can be extended to power-type maps on $\R^n$. We explain how. By definition, there is a $\wp$-type map $h:\R^{n-1} \to S^{n-1}$ which is strongly automorphic with respect to $G$ and a linear map $A = \lambda \mathcal{O} :\R^{n-1} \to \R^{n-1}$, where $\lambda >1$ and $\mathcal{O}$ is an orthogonal map in $\R^{n-1}$, so that $\ell \circ h = h\circ A$.

Now, $A$ can be extended to a linear map $\widetilde{A}:\R^n \to \R^n$ by extending via $t\mapsto \pm\lambda t$ in the $n$'th coordinate, where the sign is chosen to preserve orientation. Further, $h$ can be extended to a quasiregular map $\widetilde{h}:\R^n \to \R^n$ via
\[ \widetilde{h}(x_1,\ldots, x_n) = e^{\pm x_n} h(x_1,\ldots, x_{n-1}),\]
where $\pm$ is chosen to ensure $\widetilde{h}$ is orientation-preserving and hence quasiregular. Then $\widetilde{h}$ is strongly automorphic with respect to the group $G'$ which is isomorphic to $G$, and so $\widetilde{h}$ is of Zorich-type. The unique solution to the Schr\"oder equation $\widetilde{\ell} \circ \widetilde{h} = \widetilde{h} \circ \widetilde{A}$ is a uqr map of power-type, which agrees with $\ell$ on $S^{n-1}$. Finally, conjugating everything in the Schr\"oder equation by $g$, we obtain a new Schr\"oder equation. Hence $f$ agrees with the power-type map $g^{-1} \circ \widetilde{\ell} \circ g$ on $J(f)$.

\item For the final case where $J(f) = \Lambda (f)$ is an $(n-1)$-quasidisk for $n\geq 4$, we first conjugate $f$ by a quasiconformal map $g$ so that $f_1 = g\circ f \circ g^{-1}$ has $J(f_1) = D_{n-1}$. 
Next, Lemma \ref{lem:getp} yields a uqr map $P$ with $J(P) = S^{n-1}$ and $f_1 \circ h_1 = h_1 \circ P$. Since the branch set $B(h_1)$ of $h_1$ is not dense in $S^{n-1}$ by Proposition \ref{prop:branch}, $h_1$ is locally quasiconformal on an open dense subset of $S^{n-1}$. Therefore by Lemma \ref{lem:dist}, $\Lambda(P)$ contains an open and dense subset of $J(P)$. 

Using the arguments of \cite[section 6]{MM}, a uqr map is $\mu$-rational for some measurable conformal structure $\mu$. Since $\mu|_{S^{n-1}}$ is measurable, measurable functions are almost everywhere continuous in measure and $\Lambda(P)$ is open and dense in $J(P) = S^{n-1}$, it follows by \cite[Theorem 6.1]{MM} that $P|_{S^{n-1}}$ is a Latt\'es-type map. By applying the previous case, $P|_{S^{n-1}} = \alpha|_{S^{n-1}}$, where $\alpha$ is a power-type map in $\R^n$. Since $h_1$ semi-conjugates between power-type maps and Chebyshev-type maps via $C\circ h_1 = h_1\circ \alpha$, we may conclude that $f$ agrees with the Chebyshev-type map $g^{-1} \circ C \circ g$ on $J(f)$.
\end{enumerate}
\end{proof}

\section{A Denjoy-Wolff Theorem in dimension $3$}

In this section, we explore the converse situation of the previous sections. If $U$ is a forward invariant subset of $F(f)$ that is a quasiball, what can we say about $f|_U$? More generally, we will aim to classify the behaviour of the iterates of $f$ if $f:\B^n \to \B^n$ is uqr.

We first point out that uqr mappings on $\B^n$ need not be hyperbolic contractions, and can in fact distort the hyperbolic metric by an arbitrarily large factor.

\begin{proposition}
\label{prop:large}
Let $x_0 \in \B^n$.
For every $\lambda >0$ there exists a uqr map $f:\B^n \to \B^n$ so that in a neighbourhood of $x_0$, $f(x) = f(x_0) + \lambda (x-x_0)$. In particular $|f'(x_0)|/(1-|f(x_0)|^2)$ can be made arbitrarily large.
\end{proposition}

\begin{proof}
Let $P:\B^n \to \B^n$ be the restriction of a degree $d$ power map to the unit ball. We may assume $x_0$ is outside the branch set of $P$, otherwise conjugate $P$ by a M\"obius map. Choose $r>0$ small enough so that $P$ is injective on $U = B(x_0,r)$ and every orbit passes through $U$ at most once. We will modify $P$ on $U$ as follows.

Given $\lambda >0$, let $g(x) = \lambda(x-x_0) + P(x_0)$ be defined on $B(x_0,\epsilon)$, with $\epsilon >0$ chosen small enough so that $\overline{g(B(x_0,\epsilon))} \subset P(B(x_0,r/2))$.

We then define $f$ by setting it equal to $P$ in $\B^n \setminus B(x_0,r)$, $g$ in $B(x_0,\epsilon)$ and interpolate in between by a quasiconformal map guaranteed by Sullivan's Annulus Theorem (see for example \cite{TV}) applied to $P$ and $g$ in $B(x_0,r)$. Since every orbit of $f$ passes through $B(x_0,r)$ at most once, $f$ is uniformly quasiregular. 
\end{proof}

We will prove Theorem \ref{thm:3} below, which is a version of the Denjoy-Wolff Theorem in dimension three. However, most of our set-up applies to higher dimensions, and so we will state the preliminary results for any dimension.

\begin{lemma}
\label{lem:dw1}
Let $n\geq 2$ and $f:\B^n \to \B^n$ be a uqr map. If there is $x_0 \in \B^n$ and a subsequence $f^{m_k}$ converging locally uniformly to $x_0$, then $f^m \to x_0$ locally uniformly on $\B^n$. 
\end{lemma}

\begin{proof}
By the quasiregular version of Montel's Theorem, \cite[Theorem 4]{Miniowitz}, the family of iterates $\{f^m :m\in \N \}$ forms a normal family. The proof then follows verbatim from \cite[Proposition 4.6]{HMM}.
\end{proof}

\begin{lemma}
\label{lem:dw2}
Let $n\geq 2$ and $f:\B^n \to \B^n$ be a uqr map. Then either there is a point $x_0\in \overline{\B^n}$ and a subsequence $f^{m_k}$ so that $f^{m_k} \to x_0$ locally uniformly on $\B^n$, or the set of limit functions of subsequences of $f^m$ forms a semi-group of quasiconformal automorphisms of $\B^n$. In the latter case, $f$ itself must be quasiconformal.
\end{lemma}

\begin{proof}
Again noting that the family of iterates forms a normal family, the proof then follows verbatim from \cite[Proposition 4.9]{HMM}.
\end{proof}

In \cite{HMM}, the authors note that in the case of a parabolic basin for a uqr map, they did not exclude the possibility of more than one fixed point on the boundary of the basin arising as a local uniform limit of a subsequence of iterates. This difficulty is the main obstruction to proving higher dimensional generalizations of the Denjoy-Wolff Theorem.

Before proving our next lemma, we will require the following coarse Lipschitz result for quasiregular maps.

\begin{lemma}[\cite{V1} Theorem 11.2] 
\label{lem:lip}
Let $n\geq 2$ and let $f:\B^n\to \B^n$ be $K$-quasiregular. If $d_h$ denotes the hyperbolic distance on $\B^n$, then for any $x,y\in \B^n$ we have
\[d_h(f(x),f(y)) \leq K_I(f)(d_h(x,y)+\ln 4),\]
recalling $K_I(f)$ is the inner distortion.
\end{lemma}

\begin{lemma}
\label{lem:dw3}
Let $K\geq 1$, $n\geq 2$ and $f:\B^n\to \B^n$ be a $K$-uqr map. If there are two distinct points $x_0,x_1 \in \partial \B^n$ and two subsequences $f^{m_k},f^{p_k}$ with $f^{m_k}\to x_0$ and $f^{p_k} \to x_1$ locally uniformly on $\B^n$, then there is a continuum $C \subset S^{n-1}$ containing $x_0$ and $x_1$, and consisting of local uniform limits of subsequences of iterates of $f$.
\end{lemma}

\begin{proof}
Towards a contradiction, assume that such a continuum does not exist. Denote by $X$ the set in $S^{n-1}$ consisting of local uniform limits of subsequences of iterates of $f$, that is, $X = \{ \text{limit points of } f^n(0) \}$. Then clearly $X$ is closed in $S^{n-1}$. Consequently, $S^{n-1} \setminus X$ is open in $S^{n-1}$, and separates $X$ in the sense that $X$ contains at least two components, one containing $x_0$ and one containing $x_1$.

Choose a connected compact subset $Y$ of $S^{n-1}\setminus X$ which still separates $x_0$ and $x_1$ and extend $Y$ to $Y'$ radially inside $\B^n$ by Euclidean distance $\epsilon >0$ in such a way that $Y' \cap \{ f^m(0) : m\in \N \} = \emptyset$. Note that if such an $\epsilon$ cannot be found, then we contradict the fact that $Y \subset S^{n-1} \setminus X$. By construction, $Y'$ separates $x_0$ and $x_1$ in $\{x : 1-\epsilon \leq |x| \leq 1 \}$. Let $Y_0,Y_1$ denote the components of the complement of $Y'$ containing $x_0,x_1$ respectively and set $\delta >0$ to be the Euclidean distance between them, that is,
\begin{equation}
\label{eq:dw3eq2} 
\delta = \inf \{ |x-y | : x \in Y_0, y\in Y_1 \}.
\end{equation}
See Figure \ref{fig:6}.

\begin{figure}[h]
\begin{center}
\includegraphics[width=5in]{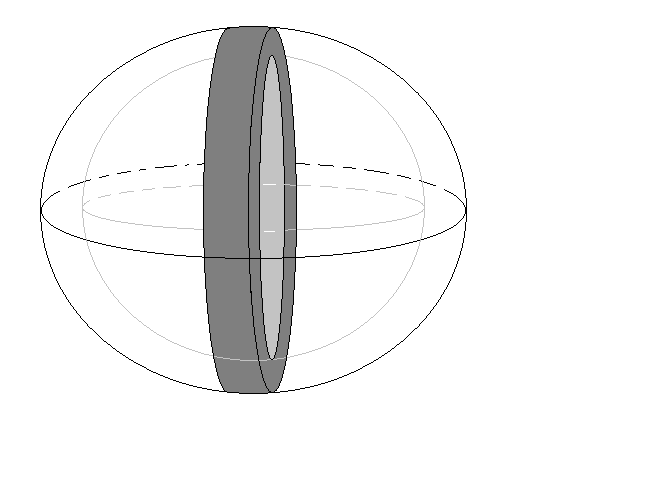}
\caption{The set $Y'$ in $\overline{\B^3}$.}
\label{fig:6}
\end{center}
\end{figure}

Next, find $N_0$ so that $|f^m(0)|> 1-\epsilon$ for $m\geq N_0$. Note that there cannot be a bounded subsequence of $f^m(0)$ with respect to the hyperbolic metric  because then there would be a fixed point of $f$ in $\B^n$, which is not the case. By passing to subsequences if necessary, assume that $|f^{m_k}(0) - x_0| < 1/k$ and $|f^{p_k}(0) - x_1| <1/k$ for all $k\in \N$. By Lemma \ref{lem:lip}, since $f$ is $K$-uqr there is a constant $M = M(K)$ such that
\[ d_h(f^{m+1}(0),f^m(0)) \leq M\]
for all $m\geq 1$. Hence as $m\to \infty$, by comparison of the hyperbolic and Euclidean metrics near the boundary of the ball, we have
\[ |f^{m+1}(0) - f^m(0)| \to 0\]
as $m\to \infty$. Choose $N_1$ so that 
\begin{equation}
\label{eq:dw3eq1}
|f^{m+1}(0) - f^m(0)| < \delta /2
\end{equation} 
for $m \geq N_1$. Then if $N_2 = \max\{N_0,N_1\}$, and if $m\geq N_2$, the sequence $f^m(0)$ is contained in $\{x : 1-\epsilon < |x| <1 \}$ and contains subsequences which intersect $Y_0$ and $Y_1$. Fix $k\in \N$ so that $f^{m_k}(0) \in Y_0$, $f^{p_k}(0) \in Y_1$ and without loss of generality assume $m_k < p_k$. Then the collection of points $\{ f^j(0) : m_k < j <p_k \}$ must pass through $Y'$ by \eqref{eq:dw3eq2} and \eqref{eq:dw3eq1}. This is a contradiction, and so $X$ must contain $x_0$ and $x_1$.
\end{proof}

\begin{lemma}
\label{lem:dw4}
Let $K\geq 1$, $n\geq 2$ and $f:\B^n \to \B^n$ a proper, surjective $K$-uqr map. Then $f$ extends to a map $\widetilde{f}:\overline{\R^n} \to \overline{\R^n}$ and the restriction of $\widetilde{f}$ to $S^{n-1}$ is $K$-uqr.
\end{lemma}

\begin{proof}
We may extend $f$ to a map $\widetilde{f}: \overline{\R^n} \to \overline{\R^n}$ by \cite[Theorem VII.3.16]{Rickman}. Since $f$ is uqr, the extension is too.
Since $S^{n-1}$ is completely invariant under $\widetilde{f}$, the restriction to $S^{n-1}$ is a map of the same degree as $f$.
With slight abuse of notation, from here we denote the restriction of $\widetilde{f}$ to $S^{n-1}$ by $f$.

Since $\widetilde{f}$ is quasiregular, its restriction $f$ also has bounded distortion. To see this, apply Theorem \ref{thm:lindist} to $\widetilde{f}$ at $x_0 \in S^{n-1}$ and obtain the same conclusion for $f$ at $x_0$. It further follows that $f^m$ has uniformly bounded distortion over all $m\in \N$.
Via the metric definition of quasiregular maps, see for example \cite[II.6]{Rickman}, we are then done if we know that $f$ is orientation preserving. We remark that this needs to be shown, since, for example, \cite{Mayer2} contains Chebyshev-type examples of uqr mappings in $\R^n$ which, when restricted to a subset of a hyperplane on which they are completely invariant, are not always orientation preserving. 

To this end, and using notation from, for example, \cite{DK}, consider the long exact sequence of the homology of the pair $(\overline{\B^n},S^{n-1})$ and the map between $\Z$-modules $f^*$ induced by $f$:

\[\begin{tikzcd}  \cdots \arrow{r} & H_n(\overline{\B^n}) \arrow{r}{j^*} & H_n(\overline{\B^n},S^{n-1}) \arrow{r}{\partial}\arrow[swap]{d}{f^*} & H_{n-1}(S^{n-1}) \arrow{r}{i^*}\arrow[swap]{d}{f^*} & H_{n-1}(\overline{\B^n}) \arrow{r} & \cdots\\
\cdots \arrow{r} & H_n(\overline{\B^n}) \arrow{r}{j^*} & H_n(\overline{\B^n},S^{n-1}) \arrow{r}{\partial} & H_{n-1}(S^{n-1}) \arrow{r}{i^*} & H_{n-1}(\overline{\B^n}) \arrow{r} & \cdots
\end{tikzcd}\]

Since $\overline{\B^n}$ is contractible, then $H_k(\overline{\B^n}) = 0$ for any $k$ and $n$.  Further, $H_{n-1}(S^{n-1}) \simeq \Z$.  Hence, 
\[0 = \text{im} j^* = \ker\partial\]
and 
\[ \text{im} \partial = \ker i^* = \Z.\]
So, $\partial$ is an isomorphism of $\Z$-modules.  For $x$ a generator of $H_n(\overline{\B^n}, S^{n-1})$, we need to know that
$\deg(f_{S^{n-1}}) \partial (x) \in H_n(S^{n-1})$ satisfies $\deg (f_{S^{n-1}}) > 0$.  

To see this, note that $H_n(\overline{\B^n}, S^{n-1}) = H_n(C(\overline{\B^n})/C(S^{n-1}))$, where $C(\overline{\B^n})/C(S^{n-1})$ is a quotient of the chain complexes on the closed ball and $S^{n-1}$.  Since $f$ is orientation-preserving on the interior of $\overline{\B^n}$, then $f^*$ maps a generator $x$ of $H_n(\overline{\B^n}, S^{n-1})$ to $(\deg f_{\B^n})x$, with $\deg f_{\B^n} > 0$.   Since $f^*$ commutes with $\partial$, then $f^*\partial(x) = \partial(\deg f_{\B^n}(x))$.  Since $\partial$ is an isomorphism, then $\partial(x)$ is a generator of $H_{n-1}(S^{n-1})$, and $\deg f_{S^{n-1}}(\partial(x)) =  \partial(\deg f_{\B^n}x)  = \deg f_{\B^n}(\partial(x))$.  Hence, $f$ restricted to $S^{n-1}$  is orientation-preserving.  
\end{proof}

We are now in a position to prove Theorem \ref{thm:3}.

\begin{proof}[Proof of Theorem \ref{thm:3}]
Let $f:\B^3 \to \B^3$ be a surjective proper uqr map. By Lemma \ref{lem:dw2}, either the family of iterates of $f$ forms a semi-group of automorphisms of $\B^3$, or there is a point $x_0 \in \overline{\B^3}$ and a subsequence $f^{m_k}$ which converges locally uniformly on $\B^3$ to $x_0$.

If $x_0 \in \B^3$, then by Lemma \ref{lem:dw1}, $f^m \to x_0$ locally uniformly on $\B^3$. 
Otherwise $x_0 \in \partial \B^3$. Assume there is another point $x_1 \in \partial \B^3$ arising as a locally uniform limit of a convergent subsequence of iterates of $f$.  Then by Lemma \ref{lem:dw3}, there is a continuum $C\subset \partial \B^3$ of such points containing $x_0$ and $x_1$. Since $f$ extends continuously to $S^2 = \partial \B^3$, this continuum $C$ consists of fixed points of $f$.

By Lemma \ref{lem:dw4}, the restriction of $f$ to $S^2$ is uqr. Since every uqr map of $S^2$ to itself is the quasiconformal conjugate of a rational map, see e.g. \cite{H}, we have $f = \varphi^{-1} \circ R \circ \varphi$ for some quasiconformal map $\varphi$ and rational map $R$. Then $\varphi(C)$ is a continuum of fixed points of $R$. We obtain a contradiction since a non-constant rational map that is not the identity cannot have a continuum of fixed points by the Identity Theorem, and our map is neither constant nor the identity.
\end{proof}

\end{document}